\documentclass[12pt,reqno]{amsart}
\usepackage[numbers,sort&compress]{natbib}
\usepackage{amsfonts,bbm}
\usepackage{amssymb,color}
\usepackage{amssymb}

\usepackage{enumitem}
\usepackage[colorlinks=true,backref]{hyperref}
\hypersetup{urlcolor=blue, citecolor=blue, linkcolor=black}

\usepackage[margin=3cm, a4paper]{geometry}

\def\normo#1{\left\|#1\right\|}

\def\abs#1{|#1|}
\def\aabs#1{\left|#1\right|}

\def\brk#1{\left(#1\right)}
\def\fbrk#1{\{#1\}}
\def\mbrk#1{\left[#1\right]}
\def\rev#1{\frac{1}{#1}}
\def\half#1{\frac{#1}{2}}
\def\norm#1{\|#1\|}
\def\jb#1{\langle#1\rangle}
\def\wt#1{\widetilde{#1}}

\def\<{\langle}
\def\>{\rangle}
\def\loe{\leqslant}
\def\goe{\geqslant}
\def\lsm{\lesssim}
\def\gsm{\gtrsim}

\def\diff{\Delta}
\def\pd{\partial}
\def\pdt{\partial_{t}}
\def\pdk{\partial_{k}}
\def\pdj{\partial_{j}} 
\def\dive{\text{div}} 
\def\dx{\text{\ dx}}
\def\ds{\text{\ ds}}
\def\dt{\text{\ dt}}

\def\lsmb{\lesssim_{\beta}}

\def\ep{\varepsilon}

\def\ph{\varphi}
\def\th{\theta}

\def\bbr{\mathbb{R}}
\def\bbz{\mathbb{Z}}

\def\ut{u_{t}}
\def\uo{u_{0}}
\def\ua{u_{1}}
\def\chir{\chi_{R}}
\def\uj{P_j U}
\def\uk{P_k U'}
\def\uja{P_{j_1}U}
\def\ujb{P_{j_2}U''}
\def\SI{S(I)}
\def\ZI{Z(I)}
\def\WSI{\widetilde{S}(I)}
\def\WSN{\widetilde{S}}
\def\half{\frac{1}{2}}
\def\qua{\frac{1}{4}}
\def\epa{\varepsilon_{1}}

\def\epo{\varepsilon_{0}}
\def\ioa{\iota_1}
\def\iob{\iota_2}
\def\hra{H^{1}_{\text{rad}}}
\def\lra{L^{2}_{\text{rad}}}

\def\taa{\tau_{1}}
\def\tab{\tau_{2}}
\def\Ta{T_{0}}
\def\pk{P_{k}}
\def\xj{x_{j}}
\def\xk{x_{k}}

\def\Tres{T_{\text{Res}}}

\renewcommand{\sigma}{\omega}

\newcommand{\R}{{\mathbb R}}
\newcommand{\C}{{\mathbb C}}

\newcommand{\ra}{{\rightarrow}}
\newcommand{\K}{{\mathcal{K}}}

\def\jb#1{\langle#1\rangle}
\def\norm#1{\|#1\|}
\def\normo#1{\left\|#1\right\|}

\def\wt#1{\widetilde{#1}}

\def\aabs#1{\left|#1\right|}

\newcommand{\F}{\mathcal{F}}
\newcommand{\LL}{\mathcal{L}}

\newcommand{\al}{\alpha}
\newcommand{\be}{\beta}

\newcommand{\de}{\delta}

\newcommand{\la}{\lambda}

\newcommand{\ta}{\tau}
\newcommand{\ka}{\kappa}

\newcommand{\De}{\Delta}

\newcommand{\Om}{\Omega}

\newcommand{\I}{\infty}

\newcommand{\EQ}[1]{\begin{align*}\begin{split} #1 \end{split}\end{align*}}
\newcommand{\EQn}[1]{\begin{align}\begin{split} #1 \end{split}\end{align}}

\newcommand{\Del}[1]{}

\numberwithin{equation}{section}

\newtheorem{thm}{Theorem}[section]
\newtheorem{cor}[thm]{Corollary}
\newtheorem{lem}[thm]{Lemma}
\newtheorem{prop}[thm]{Proposition}

\theoremstyle{remark}
\newtheorem{rem}[thm]{Remark}

\begin{document}
\title[Quadratic Klein-Gordon equation]{Scattering for the quadratic Klein-Gordon equations}
\subjclass[2010]{}

\author{Zihua Guo}
\address{(Z. Guo) School of Mathematical Sciences, Monash University, Melbourne, VIC 3800, Australia}
\email{zihua.guo@monash.edu}
\thanks{}

\author{Jia Shen}
\address{(J. Shen) School of Mathematical Sciencs, Peking University, No 5. Yiheyuan Road, Beijing 100871, P.R.China}
\email{shenjia@pku.edu.cn}
\thanks{}

\begin{abstract}
We study the scattering problems for the quadratic Klein-Gordon equations with radial initial data in the energy space. For 3D, we prove small data scattering, and for 4D, we prove large data scattering with mass below the ground state. 
\end{abstract}

\maketitle

\tableofcontents

\section{Introduction}
In this paper we study the Cauchy problems to the following quadratic Klein-Gordon equation
\EQn{ \label{main}
	\pdt^{2} u - \diff u + u  =& u^{2}, \quad \brk{t,x}\in \R\times \R^d\\
	u(0,x)  =&  \uo,\\
	\pdt u(0,x) =& \ua,
}
where $u(t,x):\bbr\times\bbr^{d} \rightarrow \bbr$, $d=3,4$.  The Klein-Gordon equation with various types of nonlinear terms ($u^2$ replaced by $f(u)$) has been extensively studied in a large amount of literatures, for example, see \cite{kg-focusing} and references therein for the detailed introduction. In particular, the existence of global solutions and study of their asymptotic behaviour are two important topics.

We first review the cases with the power type nonlinearity $f(u)=\la |u|^pu$. There are two special indices for $p$: mass-critical index $p=4/d$ and energy-critical index $p=4/(d-2)$. In view of the current studies,  when $4/d < p \loe 4/(d-2)$ ($p>4/d$ for $d=1,2$), the scattering problems were better understood. For the defocusing case $\la>0$, see \cite{brenner1984space, ginibre1985global, kg-subcritical-1-2, kg-nakanishi-energy-critical, nakanishi2001remarks} and for the focusing case $\la<0$, see \cite{kg-focusing, ibrahim2014threshold}. 
For small data, one can have scattering in critical space $H^{s}$ (see \cite{wang1998existence,wang1999scattering}). When $p\loe 4/d$, there are less results on the scattering problems in energy space. For the mass-critical case $p=4/d$, it was observed in \cite{nakanishi2008transfer} that the scattering results for Klein-Gordon equation can imply the same results for mass critical nonlinear Schr\"{o}dinger equation (NLS). On the other hand, scattering for the 2D cubic Klein-Gordon was established in \cite{killip2012scattering} using the result for NLS in \cite{dodson2015global,dodson2016global}. When $p<4/d$, the scattering results were usually obtained for small data in some weighted Sobolev sapce, for example,  in \cite{strauss1981nonlinear} for $p_{S}(d)<p \loe 4/d$, where  $p_{S}(d)$ is the Strauss exponent satisfying $dp(p+1)=2(p+2)$, and in \cite{hayashi2009scattering} for $p>2/d$. When $0<p\loe 2/d$ if $d\goe2$, or $p=3$ if $d=1$, scattering operator does not exist, see \cite{glassey1973asymptotic, matsumura1976asymptotic, georgiev1996asymptotic, delort2001existence}. 

The quadratic term $u^2$ may be compared with $|u|u$ (namely $p=1$). It is mass-subcritical for 3D and mass-critical for 4D. However, due to the better regularity and algebraic structure of $u^2$, some new methods were developed to study the asymptotic behaviour. Let us mention Klainerman's vector field method \cite{klainerman1985global} and Shatah's normal form method \cite{shatah1985normal}. Both methods showed scattering of global small solutions for \eqref{main} with $d=3$. 
For 2D, the global existence of small solutions and asymptotic behaviour were studied in \cite{ozawa1996global, delort2004global}. Note that the above two models are below the Strauss exponent, i.e. $p=2=p_S(3)$, and $p=2<p_S(2)$. The above results are for small data with sufficient regularity and decay (in weighted Sobolev space). Using the space-time resonance structure and $U^p,V^p$ space, Schottdorf  \cite{schottdorf2012global} showed small data scattering  in energy space for 3D quadratic Klein-Gordon equation. Recently, in \cite{guo2012small-za}, the first author and Nakanishi used a new approach to show the scattering for the 3D Zakharov system with small radial energy data. The idea is to combine the radially improved Strichartz estimates in \cite{guo2014improved} and (partial) normal form method in \cite{shatah1985normal}.  It turns out that this approach can deal with the scattering problems for a class of 3D quadratic dispersive equations and has been further extended. For example, see \cite{guo2012small-kgz, guo2014generalized, guo2018scattering-GP} for non-radial version generalization and applications to other equations. 

The purpose of this paper is to study the asymptotic behaviour for the quadratic Klein-Gordon equation \eqref{main} using this approach. Comparing to the $U^p,V^p$ space methods used in \cite{schottdorf2012global}, we used only Strichartz space that allows perturbation.  This gives us the possibility to study the large data problem as \cite{guo2013global-za, guo2014global-kgz}.  
Our first result is the small data scattering in energy space for the quadratic Klein-Gordon equation (\ref{main}) in 3D and 4D.
\begin{thm} \label{thm-small}
	Let $d=3$ or $d=4$, and $\ka>0$ be a sufficiently small constant. Suppose that $\brk{\uo,\ua}$ is radial, and satisfies
	\begin{equation*}
		\normo{\brk{\uo,\ua}}_{H^{1}\times L^{2}} \ll 1,
	\end{equation*}
	then there exists a unique solution $u(t,x)$ to (\ref{main}) in
	\EQ{
		C\brk{\bbr:H^1} \cap \brk{\half, \frac{3}{10}-\ka, \frac{2}{5}-3\ka|\frac{7}{10}+\ka}_\bbr, \text{\ when\ } d=3,}
	and in
	\EQ{
		C\brk{\bbr:H^1} \cap \brk{\half, \frac{5}{14}-\ka, \frac{3}{7}-4\ka|\frac{11}{14}+\ka}_\bbr, \text{\ when\ } d=4.
	}
Moreover, scattering holds, namely, $\exists\ u_{\pm}(x)\in H^{1}$ such that
\EQ{
		\normo{u-i\<D\>^{-1}\pdt u- e^{it\<D\>} u_{\pm}}_{H^{1}} \rightarrow 0, \quad t \to \pm \infty.
}
\end{thm}
\begin{rem}

\quad

(a) The notation $\brk{1/q,1/r, s_{0}|s_{1}}_I$ is the space given in \eqref{eq:space-time}.
	
(b) The scattering part in the above Theorem is not new, but we can obtain stronger results that the solutions belong to a set of perturbed Strichartz spaces, see Proposition \ref{perturbed-3d} and Proposition \ref{perturbed-4d} below. This enable us to study large data scattering.

(c) The radial assumption could be replaced by additional angular regularity by the similar arguments in \cite{guo2014generalized}.
\end{rem}

Now we turn to the large data problem. On one hand, 
the quadratic Klein-Gordon equation \eqref{main} has a conservation of energy
\EQ{
	E(u(t),\ut(t)) = \int_{\bbr^{d}} \half |\pdt u(t,x)|^{2} + \half |\nabla u(t,x)|^{2} + \half |u(t,x)|^{2} -  \frac{1}{3}u(t,x)^{3}\ dx.
}
On the other hand, 
the ground state $Q$, that is the unique radial positive solution to the elliptic equation
\begin{equation*}
	-\diff Q + Q = Q^{2},
\end{equation*}
is a stationary solution to \eqref{main}, which is non-scattering. 
It is well known that $Q$ attains the best constant of Gagliardo-Nirenberg inequality
\begin{equation*}
	\int |f(x)|^{\frac{2\brk{d+2}}{d}} \loe \frac{d+2}{d} \brk{\frac{\normo{f}_{2}}{\normo{Q}_{2}}}^{\frac{4}{d}} \normo{\nabla f}_{2}^{2}.
\end{equation*}
We want to clarify the dichotomy behaviour into blowup and scattering with $Q$ as a threshold. However, for the 3D case, we do not know how to prove scattering at the moment since it is $L^2$-subcritical and we do not have the variational analysis of the virial estimate.  So we only have the result in $4D$. 

\begin{thm}\label{thm-large}
	Let $d=4$ and $\ka>0$ be a sufficiently small constant. Suppose that $\brk{\uo,\ua}$ is radial, and satisfies
	\begin{equation*}
		E(u_{0},u_{1}) < E(Q,0).
	\end{equation*}
	\begin{enumerate} [label=(\alph*)]
		\item 
		If $\normo{u_{0}}_{2} > \normo{Q}_{2}$, the solution to (\ref{main}) blows up in finite time.
		\item
		If $\normo{u_{0}}_{2} < \normo{Q}_{2}$, the solution $u(t,x)$ to (\ref{main}) satisfies
		\EQ{
			u(t,x) \in C\brk{\bbr:H^1} \cap \brk{\half, \frac{5}{14}-\ka, \frac{3}{7}-4\ka|\frac{11}{14}+\ka}_\bbr,
		}
		and 
		\begin{equation*}	\normo{u-i\<D\>^{-1}\pdt u- e^{it\<D\>} u_{\pm}}_{H^{1}} \rightarrow 0,
		\end{equation*}
		when $t \rightarrow \pm \infty$, for some $u_{\pm}(x)\in H^{1}$.
	\end{enumerate}
\end{thm}

\begin{rem}

\quad

(a) Recently, Dodson and Murphy gave a new proof of the scattering for the focusing $H^{1/2}$-critical NLS in \cite{dodson2017new-radial}, in which they used the Virial/Morawetz estimate (used in \cite{ogawa1991blow} by Ogawa and Tsutsumi) to avoid the concentration compactness argument. We follow their idea to prove the large data scattering. 

(b) There are some difficulties for the Klein-Gordon equation. Virial/Morawetz estimate yields small $L^3$ norm at one large time, while we need small $L^3$ norm on a suitable large time interval. However, on one hand, Klein-Gordon equation does not have almost finite propagation of localised $L^{2}$ norm as NLS. On the other hand, Virial/Morawetz estimate cannot give the decay of the localization of $\normo{\pdt u}_{2}$, so local energy estimate inside the light cone cannot be applied. To overcome this difficulty, we use Cazenave's approach \cite{cazenave2003semilinear} to give a pointwise decay of $\normo{u-i\<D\>^{-1}\pdt u}_{3}$ after large time. 	
\end{rem}

An important similar equation is the following $\phi^{4}$ model, which arises in quantum field theory
\begin{equation} \label{phi4}
\pdt^{2} \phi - \diff \phi = \phi -\phi^{3},
\end{equation}
with the non-vanishing boundary condition $\lim_{|x|\to \infty}|\phi(x)|=1$.
It has the conservation of energy
\EQ{
\tilde{E}(\phi,\phi_{t}) = \int_{\bbr^{d}} \half |\pdt \phi(t,x)|^{2} + \half |\nabla \phi(t,x)|^{2} + \qua \brk{1-|\phi|^{2}}^{2} \ dx.
}
In \cite{kowalczyk2017kink}, Kowalczyk, Martel and Mu{\~n}oz studied the asymptotic stability of some kink solutions in dimension one. Note that under the simpler boundary condition 
\EQ{
\lim_{|x|\to \infty}\phi(t,x) = 1,
}
and let $w(t,x) = \phi(t,x) - 1$, then the equation (\ref{phi4}) can be transformed into Klein-Gordon equation
\begin{equation} \label{phi4-new}
\pdt^{2} w - \diff w + 2w = -3 w^{2} - w^{3}.
\end{equation}
It can be viewed as a quadratic Klein-Gordon equation perturbed with a cubic term. Note that $w^3$ term is $\dot{H}^{1/2}$ critical in 3D and $\dot{H}^{1}$ critical in 4D, so one may control the cubic term with $L_t^3L_x^6$ norm directly. Then, the small data scattering results in Theorem \ref{thm-small} also hold for \eqref{phi4-new}. It seems interesting to study the large data problem.

\medskip

\noindent{\bf Notations}

$\bullet$ $\hat f$ or $\F f$ denotes the Fourier transform of $f$.

$\bullet$ $C>0$ denotes some constant, and $C(a)>0$ denotes some constant depending on coefficient $a$.

$\bullet$ If $f\loe C g$, we write $f\lsm g$. If $f\loe C g$ and $g\loe C f$, we write $f\sim g$. Suppose further that $C=C(a)$ depends on $a$, then we write $f\lsm_a g$ and $f\sim_a g$, respectively.

$\bullet$ For $x\in\bbr^d$, $\jb{x} := \brk{1 + |x|^2}^{1/2}$.

$\bullet$ $D := \F^{-1}|\xi|\F $ and $\jb{D}^s:=\F^{-1}(1+|\xi|^2)^{s/2}\F $.  

$\bullet$ Take a cut-off function $\chi\in C_{0}^{\infty}(0,\infty)$ such that $\chi(r)=1$ if $r\loe1$ and $\chi(r)=0$ if $r>2$. For $k\in\bbz$, let $\chi_k(r) = \chi(2^{-k}r)$ and $\phi_k(r) =\chi_k(r)-\chi_{k-1}(r)$. We define the Littlewood-Paley dyadic operator $P_{\loe k} f := \mathcal{F}^{-1}\brk{ \chi_k(|\xi|) \hat{f}(\xi)}$ and $\pk f := \mathcal{F}^{-1}\brk{ \phi_k(|\xi|) \hat{f}(\xi)}$.

$\bullet$ $L^p(\bbr^d)$ and $H^s(\bbr^d)$ denote the usual Lebesgue and Sobolev space. $L^p_{\text{rad}}$(or $H^s_{\text{rad}}$) denotes the space of radial functions in $L^{p}(\bbr^d)$(or $H^s(\bbr^d)$). 

$\bullet$ $B^{s}_{p,q}$ and $\dot{B}^{s}_{p,q}$ denote the standard inhomogeneous and homogeneous Besov space, that is $\normo{f}_{B^{s}_{p,q}} := \normo{P_{\loe 0} f}_{p} + \brk{\sum_{k\goe 0} 2^{qsk}\normo{\pk f(x)}_{p}^{q}}^{1/q}$, and
$\normo{f}_{\dot{B}^{s}_{p,q}} := \brk{\sum_{k\in \bbz} 2^{qsk}\normo{\pk f(x)}_{p}^{q}}^{1/q}$ and we write $\dot{B}^{s}_{p}=\dot{B}^{s}_{p,2}$.

$\bullet$
We define the following Besov type space
\[
\normo{f}_{\brk{\dot{B}^{s_{0}}_{p}|\dot{B}^{s_{1}}_{p}}} := \brk{\sum_{k\in \bbz, k\loe0} 2^{2s_{0}k}\normo{\pk f(x)}_{p}^{2}}^{\frac{1}{2}} + \brk{\sum_{k\in \bbz, k\goe 0} 2^{2s_{1}k}\normo{\pk f(x)}_{p}^{2}}^{\frac{1}{2}}.
\]

$\bullet$
We define the norms of space-time function space
\EQn{\label{eq:space-time}
\normo{F}_{\brk{\frac{1}{q},\frac{1}{r}, s}_I} &:= \normo{F(t,x)}_{L_{t}^{q}\brk{I:\dot{B}_{r}^{s}}}, \\
\normo{F}_{\brk{\frac{1}{q},\frac{1}{r}, s_{0}|s_{1}}_I} &:= \normo{F(t,x)}_{L_{t}^{q}\brk{I:\brk{\dot{B}^{s_{0}}_{r}|\dot{B}^{s_{1}}_{r}}}}.
}
Sometimes we omit the interval $I$ for abbreviation.

\section{Small energy scattering and perturbed Strichartz estimates}

Before starting our proof, we make some preliminaries. First, we need the radially improved Strichartz estimates.
\begin{lem}[\cite{guo2018scattering-GP}]
	Suppose that $\ph\in L^2$ is radial, $d\goe2$, $2\loe q, r \loe +\infty$, and $\brk{q,r,d} \ne \brk{2, \infty, 2}$.
	If $\brk{q,r}$ satisfies $1/q + \brk{d-1}/r < \brk{d-1}/2$, we have
	\begin{equation}
	\normo{e^{it\<D\>}\pk\ph}_{\brk{\frac{1}{q},\frac{1}{r}, \frac{2}{q}+\frac{d}{r}-\frac{d}{2}| \frac{1}{q}+\frac{d}{r}-\frac{d}{2}}} \lsm \normo{\pk\ph}_{2},
	\end{equation}
	and if $\brk{q,r}$ satisfies
	$1/q + \brk{d-1}/r > \brk{d-1}/2$ and $2/q + \brk{2d-1}/r < \brk{2d-1}/2$,
	\begin{equation}
	\normo{e^{it\<D\>}\pk\ph}_{\brk{\frac{1}{q},\frac{1}{r}, \frac{2}{q}+\frac{d}{r}-\frac{d}{2}| \frac{d}{2} -1 -\frac{1}{q}-\frac{d-2}{r}}} \lsm \norm{\pk\ph}_{2}.
	\end{equation}
\end{lem}

We also gather some embedding and interpolation relations for Besov type space:
\begin{lem}
	For any $1\loe q \loe \I$, $1\loe r_2 \loe r_1 \loe \I$ and $s,s_0,s_1\in\bbr$, we have embeddings
	\EQ{
	\brk{\rev q, \rev{r_2}, s + \frac{d}{r_2} - \frac{d}{r_1} } \subset \brk{\rev q, \rev{r_1}, s},
	}
	and
	\EQ{
		\brk{\rev q, \rev{r_2}, s_0 + \frac{d}{r_2} - \frac{d}{r_1}|s_1 + \frac{d}{r_2} - \frac{d}{r_1} } \subset \brk{\rev q, \rev{r_1}, s_0|s_1}.
	}
	For $s_1 \ne s_2$, $s'_1 \ne s'_2$, $0< \th < 1$, and $1\loe q,q_1,q_2,r,r_1,r_2 \loe \I$ with
	\EQ{
	& s = (1-\th)s_1 + \th s_2, \\
	& s' = (1-\th)s'_1 + \th s'_2, \\
	& \rev q = \frac{1-\th}{q_1} + \frac{\th}{q_2}, \\
	}
	and $1/r = (1-\th)/r_1 + \th/r_2$, we have complex interpolation inequality
	\EQ{
	\norm{u}_{\brk{\rev q, \rev r, s}} \lsm \norm{u}_{\brk{\rev{q_1}, \rev{r_1}, s_1}}^{1-\th} \norm{u}_{\brk{\rev{q_2}, \rev{r_2}, s_2}}^{\th},
	}
	and
	\EQ{
		\norm{u}_{\brk{\rev q, \rev r, s|s'}} \lsm \norm{u}_{\brk{\rev{q_1}, \rev{r_1}, s_1|s'_1}}^{1-\th} \norm{u}_{\brk{\rev{q_2}, \rev{r_2}, s_2|s'_2}}^{\th}.
	}
	For $s_1 \goe s_1'$ and $s_2 \loe s_2'$, we also have trivial bound
	\EQ{
	\norm{u}_{\brk{\rev q, \rev r, s_1|s_2}} \lsm \norm{u}_{\brk{\rev q, \rev r, s'_1|s'_2}}.
	}
\end{lem}

Another important tool used in this paper is the normal form method. We first introduce a general definition of normal form. By the change of variable
\begin{equation*}
	U(t,x) = u(t,x) - i \<D\>^{-1} \ut(t,x),
\end{equation*}
we can transform the original equation into a first order one
\begin{equation} \label{1order}
i\pdt U + \<D\> U = \<D\>^{-1} u^{2} = \frac{1}{4} \<D\>^{-1} \brk{U^{2} + 2U \overline{U} + \overline{U}^{2}},
\end{equation}
then the integral equation is
\EQ{
	U(t,x)= & e^{it\jb{D}} U_{0} - \frac{i}{4} \int_{0}^{t} e^{i(t-s)\jb{D}} \<D\>^{-1} \brk{U^{2} + 2U \overline{U} + \overline{U}^{2}}  \ds. 
}
Let $m(\xi_{1},\xi_{2})$ be some Coifman-Meyer bilinear multiplier. We write $U^+(t,x) = U(t,x)$ and $U^-(t,x) = \overline{U}(t,x)$. For $(\ioa,\iob) \in \fbrk{(+,+), (+,-), (-,+), (-,-)}$, we define the normal form for different nonlinear terms as
\EQ{
	\widehat{\Om_{\ioa,\iob}} (U^{\ioa},U^{\iob}) (t, \xi) = & \int_{\bbr^{d}} \frac{1}{i\Phi(\xi,\eta)} m(\xi-\eta,\eta) \widehat{U^{\ioa}}(t,\xi-\eta) \widehat{U^{\iob}}(t,\eta) \text{d}\eta,
}
where the modulation $\Phi(\xi,\eta) := \Phi_{\ioa,\iob}(\xi,\eta)$ is defined by
\begin{equation*}
	\Phi_{\ioa,\iob}(\xi,\eta) := -\<\xi\> + \ioa \<\xi-\eta\> + \iob \<\eta\>.
\end{equation*}
Thus, the normal form transform adapted to the equation (\ref{1order}) is defined by
\EQ{
\brk{U + i\jb{D}^{-1} \Om(U,U)}(t,x) := U(t,x) + \frac{i}{4} \jb{D}^{-1}  \sum \Om_{\ioa,\iob}(U^{\ioa},U^{\iob}) (t, x),
}
where the summation is over $(\ioa,\iob) \in \fbrk{(+,+), (+,-), (-,+), (-,-)}$.
Note that the normal form is well-defined, if $|\Phi| \ne 0$ for all $\brk{\xi,\eta} \in \bbr^d\times\bbr^d$ satisfying $m(\xi-\eta,\eta)\ne 0$. 
In this paper, we are going to use the normal form with $m(\xi-\eta,\eta)$ satisfying $m(\xi-\eta,\eta) = 0$ unless $\max\fbrk{\aabs{\xi-\eta},\aabs{\eta}} \lsm 2^{-\be}$ for some large constant $\be>0$. For any choice of $(\ioa,\iob)$, the modulation $\Phi$ satisfies
\EQ{
	\aabs{\Phi(\xi,\eta)} = \aabs{\jb{\xi} \pm  \jb{\xi-\eta} \pm \jb{\eta}} \gsm \frac{1}{\jb{\min\fbrk{\abs{\xi-\eta},\abs{\eta}}}} \gsm_\be 1.
}
Therefore, we can only consider nonlinear term $U^{2}$ for simplicity, and the proof of the Strichartz estimates for other kinds of nonlinear term is essentially the same. 

In this section, we focus on the simplified equation $i\pdt U + \jb{D}U = \jb{D}^{-1}(U^2)$. For any functions $f(t,x)$ and $g(t,x)$, we define the normal form as
\EQn{\label{normal-form}
	\widehat{\Om} (f,g) (t, \xi) = & \int_{\bbr^{d}} \frac{1}{i\brk{-\jb{\xi} +  \jb{\xi-\eta} + \jb{\eta}}} m(\xi-\eta,\eta) \widehat{f}(t,\xi-\eta) \widehat{g}(t,\eta) \text{d}\eta.
}
Now we insert the normal form transform into the equation, and get
\EQ{
&\brk{i\pdt  + \jb{D}}\brk{U + i\jb{D}^{-1}\Om(U,U)} \\
= & \jb{D}^{-1}(U^2) + i\Om(U,U) \\
&+ i\Om(i\pdt U,U) + i\Om(U,i\pdt U) \\
= & \jb{D}^{-1}U^2 + i\Om(U,U) + i\jb{D}^{-1}\Om(-\jb{D}U,U) + i\jb{D}^{-1}\Om(U,-\jb{D}U) \\
&+ i\Om\brk{\<D\>^{-1}U^{2},U} + i\Om\brk{U,\<D\>^{-1}U^{2}}.
}
The quadratic term is
\EQ{
& \jb{D}^{-1}U^2 + i\Om(U,U) + i\jb{D}^{-1}\Om(-\jb{D}U,U) + i\jb{D}^{-1}\Om(U,-\jb{D}U) \\
= & \F^{-1}\brk{\jb{\xi}^{-1} \int \widehat{U}(t,\xi-\eta) \widehat{U}(t,\eta) \text{d}\eta}  \\
& + \F^{-1}\brk{ i \int \frac{1}{i\Phi} m(\xi-\eta,\eta) \brk{1 - \frac{\jb{\xi-\eta}}{\jb{\xi}} - \frac{\jb{\eta}}{\jb{\xi}}} \widehat{U}(t,\xi-\eta) \widehat{U}(t,\eta) \text{d}\eta }\\
= & \F^{-1}\brk{\jb{\xi}^{-1} \int \widehat{U}(t,\xi-\eta) \widehat{U}(t,\eta) \text{d}\eta - \jb{\xi}^{-1} \int  m(\xi-\eta,\eta)  \widehat{U}(t,\xi-\eta) \widehat{U}(t,\eta) \text{d}\eta}.
}
In fact, the Coifman-Meyer bilinear operator with multiplier $1-m(\xi-\eta,\eta)$ is the resonance term, namely
\EQ{
\F\brk{\Tres(f,g)}(\xi) := \int \brk{1 - m(\xi-\eta,\eta)}  \widehat{f}(t,\xi-\eta) \widehat{g}(t,\eta) \text{d}\eta.
}
After normal form reduction, we have
\begin{align}  \label{mainequ}
	\begin{split}
		U(t,x)= & K(t)\brk{U_{0} + i\<D\>^{-1}\Om(U,U)\brk{0}}-i\<D\>^{-1}\Om(U,U) \\
		& -i\int_{0}^{t} K(t-s) \<D\>^{-1} \Tres(U,U) \ds \\
		& +\int_{0}^{t} K(t-s) \<D\>^{-1}\brk{\Om\brk{\<D\>^{-1}U^{2},U}} \ds \\
		& +\int_{0}^{t} K(t-s) \<D\>^{-1}\brk{\Om\brk{U,\<D\>^{-1}U^{2}}} \ds.
	\end{split}
\end{align}

Finally, our normal form transform is based on frequency decomposition. Fixed a large parameter $\be>0$, for any two functions $U,U'\in H^1$, we split the decomposition as
\EQ{
U(x) U'(x) = & \sum_{(j,k)\in\bbz^2} \uj\uk \\
= & \sum_{(j,k)\in HH} \uj\uk +\sum_{(j,k)\in HL} \uj\uk + \sum_{(j,k)\in LH} \uj\uk \\
& + \sum_{(j,k)\in LL} \uj\uk,
}
where
\begin{align*}
	& HH:=\{\brk{j,k}\in \bbz^{2}: j,k\goe -\be -10\}, \\
	& HL:=\{\brk{j,k}\in \bbz^{2}: j\goe k+ 5,\ j\goe -\be -10,\ -\be+10\goe k \}, \\
	& LL:=\{\brk{j,k}\in \bbz^{2}: j,k\loe -\be +10\}.
\end{align*}
and $LH := \fbrk{(j,k)\in \bbz^2: (k,j)\in HL}$. In fact, we have
\EQn{
\bbz^2 = & HH \cup \fbrk{\brk{j,k}\in \bbz^{2}: \min\fbrk{j,k}\loe -\be -10} \\
= & HH \cup LL \cup \fbrk{\brk{j,k}\in \bbz^{2}: \min\fbrk{j,k}\loe -\be - 10, \max\fbrk{j,k}\goe -\be + 10},
}
and
\EQn{
\fbrk{\brk{j,k}\in \bbz^{2}: j\goe k, k\loe -\be - 10, j\goe -\be + 10} \subset HL.
}
For any subset $S\subset\bbz^2$, let
\EQ{
	m_S(\xi-\eta,\eta) := \sum_{(j,k)\in S} \phi_j(\xi-\eta)\phi_k(\eta),
}
and define the bilinear frequency cut-off to $S$ by
\EQ{
(UU')_S := \F^{-1} \int m_S(\xi-\eta,\eta) \widehat{U}(\xi-\eta) \widehat{U'}(\eta) \text{d} \eta = \sum_{(j,k)\in S} \uj\uk.
}
In this paper, we will take $m=m_{LL}$, so we can easily check that $m$ is bounded, and satisfies
\EQ{
\aabs{\pd_\xi^{\al} \pd_\eta^\be m(\xi-\eta,\eta)} \loe C(\al,\be) \aabs{\xi}^{-|\al|} \aabs{\eta}^{-|\be|}.
}

\subsection{3D case}

Let $\ka>0$ and $\ep>0$ be some small coefficients. In this section, we take the normal form (\ref{normal-form}) with  $m=m_{LL}$, then the resonance term is
\EQ{
	\Tres(U,U) = (UU')_{HH+LH+HL}.
} 
We also have roughly
\EQ{
	\Om(U,U') \sim \brk{UU'}_{LL}.
}
Let $\SI$ be the strong Strichartz norm
\begin{equation*}
	\SI = \brk{0,\half,0|1}\cap \brk{\half, \frac{3}{10}-\ka, \frac{2}{5}-3\ka|\frac{7}{10} + \ka}.
\end{equation*}
The interpolation space between $L_t^\I L_x^{3}$ and $L_t^\I H_x^{1}$ is defined as follows
\begin{equation*}
	\ZI = \brk{0,0,-\frac{5}{4}| -\frac{3}{4}}.
\end{equation*}
Define weak Strichartz norm $\WSI$ by
\begin{equation*}
\norm{U}_{\WSI} = \norm{P_{\goe 0}U}_{\brk{\half-\ep, \qua + \ep, \qua + 2\ep}} + \norm{P_{\loe 0}U}_{\brk{\ep, \half - 2\ep, -\ep}}  + \norm{U}_{ \brk{\frac{1}{3},\frac{1}{6},-\ep|\ep}}.
\end{equation*}
\begin{lem}[Resonance term]
	For radial $U$ and $U'$, we have
	\begin{equation}
		\normo{\int_{0}^{t} K(t-s) \<D\>^{-1} \brk{UU'}_{HH+ HL + LH} ds}_{\SI} \lsm_{\be} \normo{U}_{\WSI} \normo{U'}_{\WSI}.
	\end{equation}
\end{lem}
\begin{proof}
	Note that $\WSI$ can be interpolated by $\SI$ and $\ZI$. By interpolation, we have 
	\begin{align*}
		\normo{P_{\goe 0}U}_{\brk{\half -\ep, \qua+\ep,\qua + 2\ep}}  \lsm & \normo{P_{\goe 0}U}_{\brk{\half-\ep,\brk{1-2\ep}\brk{\frac{3}{10}-\ka}, \qua + 2\ep + 3[\brk{1-2\ep}\brk{\frac{3}{10}-\ka}-\qua -\ep]}} \\
		\lsm & \normo{P_{\goe 0}U}_{\brk{\half-\ep,\brk{1-2\ep}\brk{\frac{3}{10}-\ka}, \frac{9}{20}}} \\
		\lsm & \normo{P_{\goe 0}U}_{\brk{\half-\ep,\brk{1-2\ep}\brk{\frac{3}{10}-\ka}, \brk{1-2\ep}\brk{\frac{7}{10}+ \ka} - 2\ep \brk{-\frac{3}{4}}}} \\
		\lsm & \normo{P_{\goe 0}U}_{\SI}^{1-2\ep} \normo{P_{\goe 0}U}_{\ZI}^{2\ep},
	\end{align*}
	and
	\EQ{
	    \norm{P_{\loe 0}U}_{\brk{\ep, \half - 2\ep, -\ep}} \lsm & \norm{P_{\loe 0}U}_{\brk{\ep, (1-2\ep)\brk{\half -\ep} + 2\ep\brk{\frac{3}{10}-\ka}, -\ep + 3\mbrk{(1-2\ep)\brk{\half -\ep} + 2\ep\brk{\frac{3}{10}-\ka} - \half + 2\ep}}} \\
	    \sim & \norm{P_{\loe 0}U}_{\brk{\ep, (1-2\ep)\brk{\half -\ep} + 2\ep\brk{\frac{3}{10}-\ka}, \frac{4}{5}\ep + 6\ep^2 -6\ka\ep}} \\
	    \lsm & \norm{P_{\loe 0}U}_{\brk{0, \half -\ep, \frac{6\ep^2}{1-2\ep}}}^{1-2\ep} \norm{P_{\loe 0}U}_{\brk{\half, \frac{3}{10}-\ka, \frac{2}{5}-3\ka}}^{2\ep} \\
	    \lsm & \brk{\norm{P_{\loe 0}U}_{\brk{0, \half, 0}}^{1-2\ep} \norm{P_{\loe 0}U}_{\brk{0, 0, -\frac{5}{4}}}^{2\ep} }^{1-2\ep} \norm{P_{\loe 0}U}_{\brk{\half, \frac{3}{10}-\ka, \frac{2}{5}-3\ka}}^{2\ep} \\
	    \lsm & \normo{P_{\loe 0}U}_{\SI}^{1-2\ep(1-2\ep)} \normo{P_{\loe 0}U}_{\ZI}^{2\ep(1-2\ep)}.
	}
	
	For $\brk{j,k}\in HH$, by boundedness of Littlewood-Paley operator and H\"older inequality,
	\begin{align*}
		\normo{\uj\uk}_{\brk{1-2\ep,\half+2\ep, 2\ep|4\ep}}  \lsmb & \norm{2^{2\ep l +2\ep l^+} \norm{P_l\brk{\uj\uk}}_{\brk{1-2\ep,\half+2\ep, 2\ep|4\ep}}}_{l_l^2} \\
		\lsm_\be & 2^{5\ep\max\fbrk{j,k}} \norm{\uj\uk}_{L_t^{1/(1-2\ep)}L_x^{1/(1/2 + 2\ep)}} \\
		\lsm_{\be} & 2^{-\ep\brk{j+k}} \normo{\uj}_{\brk{\half-\ep, \qua + \ep, 6\ep}} \normo{\uk}_{\brk{\half-\ep, \qua + \ep, 6\ep}}.
	\end{align*}
	Next, we sum over $j$ and $k$:
	\begin{align*}
	\normo{\int_{0}^{t} K(t-s) \<D\>^{-1} \brk{UU'}_{HH} ds}_{\SI}   \lsm_{\be}&   \sum_{\brk{j,k}\in HH} \normo{\uj\uk}_{\brk{1-2\ep,\half+2\ep, 2\ep|4\ep}} \\
	 \lsmb & \sum_{j\goe-\be-10} 2^{-\ep j} \normo{\uj}_{\brk{\half-\ep, \qua + \ep, 6\ep}} \\
	& \times \sum_{k\goe-\be-10} 2^{-\ep k} \normo{\uk}_{\brk{\half-\ep, \qua + \ep, 6\ep}} \\
	\lsmb & \normo{P_{\goe 0}U}_{\brk{\half-\ep, \qua + \ep, 6\ep}} \normo{P_{\goe 0}U'}_{\brk{\half-\ep, \qua + \ep, 6\ep}}.
	\end{align*}
	
	For $\brk{j,k}\in HL$, 
	\EQ{
	\normo{\int_{0}^{t} K(t-s) \<D\>^{-1} \brk{UU'}_{HL} ds}_{\SI}   \lsm_{\be}&   \sum_{\brk{j,k}\in HL} \normo{\uj\uk}_{\brk{\half,\frac{3}{4} - \ep, \qua + \ep}} \\
	\lsm_{\be}&   \sum_{\brk{j,k}\in HL} 2^{\brk{\qua +\ep} j} \normo{\uj}_{\brk{\half -\ep,\frac{1}{4} + \ep, 0}} \\
	& \times \normo{\uk}_{\brk{\ep,\frac{1}{2} - 2 \ep, 0}} \\
	\lsm_{\be}&   \sum_{j\goe -\be -10} 2^{- \ep j} \normo{\uj}_{\brk{\half -\ep,\frac{1}{4} + \ep, \qua + 2\ep}} \\
	& \times \sum_{k\loe -\be + 10} 2^{\ep k} \normo{\uk}_{\brk{\ep,\frac{1}{2} - 2 \ep, -\ep}} \\
	\lsm_{\be}& \normo{P_{\goe 0}U}_{\brk{\half -\ep,\frac{1}{4} + \ep, \qua + 2\ep}} \normo{P_{\loe 0}U'}_{\brk{\ep,\frac{1}{2} - 2 \ep, -\ep}}.
	}
	By symmetry, we can also bound $LH$ term. Later, we will omit the details on the summation over $j$ and $k$.
\end{proof}
\begin{lem}[Boundary term]
	For radial $U$ and $U'$, there exists $\th>0$, such that
	\begin{equation}
		\normo{\<D\>^{-1} \Om(U,U')}_{\SI} \lsm 2^{-\th\be} \normo{U}_{\SI}^{1-2\ep}\normo{U}_{\ZI}^{2\ep} \normo{U'}_{\SI}^{1-2\ep}\normo{U'}_{\ZI}^{2\ep}.
	\end{equation}
\end{lem}
\begin{proof}
	First, we estimate $\normo{\Om(U,U')}_{\brk{0,1/2,0}}$. For $\brk{j,k}$ such that $j\loe -\beta +10$ and $k\loe -\be +10$, by Coifman-Meyer type bilinear multiplier estimate (see Lemma 3.5 in \cite{guo2012small-za}), we have
	\begin{align*}
		\normo{\Om(\uj,\uk)}_{\brk{0,\half,0}} \lsm & \normo{\uj}_{\brk{0,\qua,0}} \normo{\uk}_{\brk{0,\qua,0}} \\
		\lsm & \normo{\uj}_{\brk{0,\half-\ep,\frac{3}{4}-3\ep}} \normo{\uk}_{\brk{0,\half-\ep,\frac{3}{4}-3\ep}} \\
		\lsm & 2^{\half\brk{j+k}} 2^{-\frac{1}{8}\be} \normo{\uj}_{\brk{0,\half-\ep, -2\ep\frac{5}{4}}} \normo{\uk}_{\brk{0,\half-\ep, -2\ep\frac{5}{4}}}.
	\end{align*}
	
	As for the other norm in $\SI$, for $j\loe -\be +10$, we have interpolation
	\begin{align*}
		\normo{\uj}_{\brk{\qua,\half\brk{\frac{3}{10}-\ka}, -\frac{1}{2}}} \lsm & \normo{\uj}_{\brk{\qua,\half\brk{\frac{3}{10}-\ka} + \brk{\half-2\ep}\half, \half\brk{\frac{2}{5}-3\ka}-2\ep\frac{5}{4}}} \\
		\lsm & \normo{\uj}_{\brk{\half, \frac{3}{10}-\ka, \frac{2}{5}-3\ka}}^{\half} \normo{\uj}_{\brk{0,\half,0|1}}^{\half-2\ep} \normo{\uk}_{\brk{0,0,-\frac{5}{4}}}^{2\ep}.
	\end{align*}
	Therefore, for $\brk{j,k}\in LL$,
	\begin{align*}
		\normo{\Om(\uj,\uk)}_{\brk{\half, \frac{3}{10}-\ka, \frac{2}{5}-3\ka|\frac{7}{10} + \ka}} \lsm & \normo{\uj}_{\brk{\qua,\half\brk{\frac{3}{10}-\ka},0}} \normo{\uk}_{\brk{\qua,\half\brk{\frac{3}{10}-\ka},0}} \\
		\lsm & 2^{\qua\brk{j+k}} 2^{-\half\be} \normo{\uj}_{\brk{\qua,\half\brk{\frac{3}{10}-\ka},-\half}} \\
		& \times \normo{\uk}_{\brk{\qua,\half\brk{\frac{3}{10}-\ka},-\half}},
	\end{align*}
	then the lemma follows.
\end{proof}

\begin{lem}[Trilinear term]
	For radial $U$, $U'$ and $U''$, we have
	\begin{equation}
	\normo{\int_{0}^{t} K(t-s) \<D\>^{-1} \Om\brk{\<D\>^{-1}\brk{UU''},U'} ds}_{\SI} \lsm_{\be} \normo{U}_{\WSI} \normo{U''}_{\WSI} \normo{U'}_{\WSI}.
	\end{equation}
\end{lem}
\begin{proof}
	By interpolation, for $j\in \bbz$, we have
	\begin{align*}
		\normo{\uj}_{\brk{\frac{1}{3},\frac{1}{6},-\ep|\ep}} \lsm & \normo{\uj}_{\brk{\frac{1}{3}, \frac{2}{3}\brk{\frac{3}{10}-\ka}+\brk{\frac{1}{3}-2\ep}\half, \frac{8}{15}|\frac{11}{15}}} \\
		\lsm & \normo{\uj}_{\brk{\frac{1}{3}, \frac{2}{3}\brk{\frac{3}{10}-\ka}+\brk{\frac{1}{3}-2\ep}\half, \frac{2}{3}\brk{\frac{2}{5}-3\ka} - 2\ep\frac{5}{4}|\frac{2}{3}\brk{\frac{7}{10}+\ka} + \brk{\frac{1}{3}-2\ep} - 2\ep \frac{3}{4}}} \\
		\lsm & \normo{\uj}_{\brk{\half, \frac{3}{10}-\ka, \frac{2}{5}-3\ka|\frac{7}{10} + \ka}}^{\frac{2}{3}} \normo{\uj}_{\brk{0,\half,0|1}}^{\frac{1}{3}-2\ep} \normo{\uj}_{\brk{0,0,-\frac{5}{4}| -\frac{3}{4}}}^{2\ep}.
	\end{align*}
	By boundedness of Coifman-Meyer operator, we have
	\begin{align*}
		\normo{\Om\brk{\<D\>^{-1}\brk{UU''},U'}}_{L^{1}_{t}L^{2}_{x}} \lsm & \normo{UU''}_{L_t^{3/2}L_x^3}\normo{\normo{\uk}_{\brk{\frac{1}{3},\frac{1}{6},0}}}_{l_{k}^{1}} \\
		\lsm & \normo{UU''}_{L_t^{3/2}L_x^3}\normo{U'}_{\brk{\frac{1}{3},\frac{1}{6},-\ep|\ep}}.
	\end{align*}
	For all $j_{1}$ and $j_{2}$ in $\bbz$, we have
	\begin{align*}
		\normo{\uja\ujb}_{L_t^{3/2}L_x^3} \lsm & 2^{\ep j_{1} -2\ep j_{1}^{+}} \normo{\uja}_{\brk{\frac{1}{3},\frac{1}{6},-\ep|\ep}} \\
		& \times 2^{\ep j_{2} -2\ep j_{2}^{+}} \normo{\ujb}_{\brk{\frac{1}{3},\frac{1}{6},-\ep|\ep}}.
	\end{align*}
	Thus, the lemma follows, noting that $\WSI \subset \brk{\frac{1}{3},\frac{1}{6},-\ep|\ep}$.
\end{proof}
Combining all the above estimates, we obtain a perturbed Strichartz estimate in 3D case:
\begin{prop} \label{perturbed-3d}
	Let $d=3$, $\ep>0$ and $\ka>0$ are small constants. Assume that $U$ is a solution of (\ref{mainequ}) with initial data $U_{0} \in \hra$ , then there exists $\th>0$ such that
	\begin{equation} \label{perturbed-3d-ineq}
		\normo{U}_{\SI} \lsm \normo{U_{0}}_{H^{1}} + 2^{-\th\be} \normo{U}_{\SI}^{2\brk{1-2\ep}} \normo{U}_{\ZI}^{2\ep} + \normo{U}_{\WSI}^{2} + \normo{U}_{\WSI}^{3}.
	\end{equation}
	where the weak norm satisfies interpolation
	\EQ{
	\normo{U}_{\WSI} \lsm \norm{U}_{S(I)}^{1-2\ep + 4\ep^2} \norm{U}_{Z(I)}^{2\ep(1-2\ep)}.
	}
	Furthermore, we have small data scattering for (\ref{main}) in 3D case.
\end{prop}

\subsection{4D case}

Let $0<\ka\ll\ep\ll 1$ and $\de>0$ be some small coefficients. In this section, we take the normal form (\ref{normal-form}) with  $m=m_{LL}$, then the resonance term is
\EQ{
	\Tres(U,U) = (UU')_{HH+LH+HL}.
} 
We also have roughly
\EQ{
	\Om(U,U') \sim \brk{UU'}_{LL}.
}
Let $\SI$ be the strong Strichartz norm
\begin{equation*}
\SI = \brk{0,\half,0|1}\cap \brk{\half, \frac{5}{14}-\ka, \frac{3}{7}-4\ka|\frac{11}{14} + \ka}.
\end{equation*}
The interpolation space between $L_t^\I L_x^{3}$ and $L_t^\I H_x^{1}$ is defined as follows
\begin{equation*}
\ZI = \brk{0,0,-\half| -\frac{4}{3} +\de}.
\end{equation*}
The weak norm $\WSI$ is
\begin{align*}
	\normo{U}_{\WSI} = & \normo{P_{\goe 0}U}_{\brk{\half -\ep, \qua+\ep,7\ep} \cap \brk{\half -\ep, \qua + 3\ep, \frac{2}{7}}}  \\
	& + \normo{P_{\loe 0}U}_{\brk{\half -\ep, \qua - \ep, \ep} \cap L_t^3L_x^6 \cap \brk{\ep,2\ep\brk{\frac{5}{14}-\ka} + \brk{1-4\ep}\half, 1}}.
\end{align*}

\begin{lem}[Resonance term]
	Assume that $U$ and $U'$ are radial. For $0<\ka \ll \ep\ll 1$, we have
	\begin{equation}
	\normo{\int_{0}^{t} K(t-s) \<D\>^{-1} \brk{UU'}_{HH + HL +LH} \ds}_{\SI} \lsm_{\be} \normo{U}_{\WSI} \normo{U'}_{\WSI}.
	\end{equation}
\end{lem}
\begin{proof}
	By interpolation, for $j\goe -\be -10$, we have 
	\begin{align*}
		\normo{\uj}_{\brk{\half -\ep, \qua+\ep,7\ep}} \lsm_{\be} & \normo{\uj}_{\brk{\half-\ep,\brk{1-2\ep}\brk{\frac{5}{14}-\ka}, 7\ep + 4[\brk{1-2\ep}\brk{\frac{5}{14}-\ka}-\qua -\ep]}} \\
		\lsm_{\be} & \normo{\uj}_{\brk{\half-\ep,\brk{1-2\ep}\brk{\frac{3}{10}-\ka}, \brk{1-2\ep}\brk{\frac{11}{14} + \ka} - 2\ep \brk{-\frac{4}{3}+\de}}} \\
		\lsm_{\be} & \normo{\uj}_{\SI}^{1-2\ep} \normo{\uj}_{\ZI}^{2\ep}.
	\end{align*}
	For $\brk{j,k}\in HH$, we have 
	\begin{align*}
		\normo{\uj\uk}_{\brk{1-2\ep,\half+2\ep, 4\ep|6\ep}} \lsmb & 2^{\brk{6+\half}\ep \max\fbrk{j,k}} \norm{\uj\uk}_{L_t^{1/(1-2\ep)} L_x^{1/(1/2 + 2\ep)}} \\
		\lsm_{\be} & 2^{-\half \ep\brk{j+k}} \normo{\uj}_{\brk{\half-\ep, \qua + \ep, 7\ep}} \normo{\uk}_{\brk{\half-\ep, \qua + \ep, 7\ep}}.
	\end{align*}
	Next, consider $\brk{j,k} \in HL$, and $LH$ case follows easily. For $j\goe -\be -10$, we have interpolation
	\begin{align*}
		\normo{\uj}_{\brk{\half -\ep, \qua + 3\ep, \frac{2}{7}}} \lsm & \normo{\uj}_{\brk{\half -\ep, \brk{1-2\ep}\brk{\frac{5}{14}-\ka}, \frac{2}{7}+4[\brk{1-2\ep}\brk{\frac{5}{14}-\ka}-\qua - 3\ep]}} \\
		\lsm & \normo{\uj}_{\brk{\half, \frac{5}{14}-\ka, \frac{11}{14} + \ka}}^{1-2\ep} \normo{\uj}_{\brk{0,0,-\frac{4}{3}+\de}}^{2\ep},
	\end{align*}
	and for $k\loe-\be+10$, noting that $\ka\ll \ep$, 
	\begin{align*}
		\normo{\uk}_{\brk{\half -\ep, \qua - \ep, \ep}} \lsm & \normo{\uk}_{\brk{\half -\ep, \brk{1-2\ep}\brk{\frac{5}{14}-\ka}, \ep + 4[\brk{1-2\ep}\brk{\frac{5}{14}-\ka}-\qua + \ep]}} \\
		\lsm & \normo{\uk}_{\brk{\half -\ep, \brk{1-2\ep}\brk{\frac{5}{14}-\ka}, \frac{3}{7}  -\frac{13}{7}\ep -4\ka +8\ep\ka}} \\
		\lsm & \normo{\uk}_{\brk{\half, \frac{5}{14}-\ka, \frac{3}{7}-4\ka}}^{1-2\ep} \normo{\uk}_{\brk{0,0,-\half}}^{2\ep}.
	\end{align*}
	For $\brk{j,k} \in HL$, we have
	\begin{align*}
		\normo{\uj\uk}_{\brk{1-2\ep,\half+2\ep,4\ep|6\ep}} \lsm & \normo{\uj}_{\brk{\half-\ep,\qua+3\ep,6\ep}} \normo{\uk}_{\brk{\half-\ep,\qua-\ep,0}} \\
		\lsm & 2^{-\frac{2}{7}j + 6\ep j+\ep k}\normo{\uj}_{\brk{\half-\ep,\qua+3\ep,\frac{2}{7}}} \normo{\uk}_{\brk{\half-\ep,\qua-\ep,-\ep}}.
	\end{align*}
\end{proof}
\begin{lem}[Boundary term]
	Assume that $U$ and $U'$ are radial. For $0<\ka \ll \ep \ll 1$, there exists $\th>0$, such that
	\begin{equation}
	\normo{\<D\>^{-1} \Om(U,U')}_{\SI} \lsm 2^{-\th\be} \normo{U}_{\SI}^{1-2\ep}\normo{U}_{\ZI}^{2\ep} \normo{U'}_{\SI}^{1-2\ep}\normo{U'}_{\ZI}^{2\ep}.
	\end{equation}
\end{lem}
\begin{proof}
	First, we estimate $\normo{\Om(U,U)}_{\brk{0,\half,0}}$. For $\brk{j,k}\in LL$, we have
	\begin{align*}
		\normo{\Om(\uj,\uk)}_{\brk{0,\half,0}} \lsm & \normo{\uj}_{\brk{0,\qua,0}} \normo{\uk}_{\brk{0,\qua,0}} \\
		\lsm & \normo{\uj}_{\brk{0,\half-\ep,1-4\ep}} \normo{\uk}_{\brk{0,\half-\ep,1-4\ep}} \\
		\lsm & 2^{\qua\brk{j+k}} 2^{-\qua\be} \normo{\uj}_{\brk{0,\half-\ep, 2\ep\brk{-\half}}} \normo{\uk}_{\brk{0,\half-\ep, 2\ep\brk{-\half}}}
	\end{align*}
	
	As for the other norm in $\SI$, we have interpolation
	\begin{align*}
	\normo{\uj}_{\brk{\half-\ep, \brk{1-2\ep}\brk{\frac{5}{14}-\ka}, \frac{4}{7}}} \lsm & \normo{\uj}_{\brk{\half-\ep, \brk{1-2\ep}\brk{\frac{5}{14}-\ka}, \brk{1-2\ep}\brk{\frac{3}{7}-4\ka}+2\ep\brk{-\half}}} \\
	\lsm & \normo{\uj}_{\brk{\half, \frac{5}{14}-\ka, \frac{3}{7}-4\ka}}^{1-2\ep}\normo{\uj}_{\brk{0,0,-\half}}^{2\ep},
	\end{align*}	
	and
	\begin{align*}
		\normo{\uk}_{\brk{\ep, \brk{1-4\ep}\half + 2\ep\brk{\frac{5}{14}-\ka}, -\frac{1}{8}\ep}}  \lsm & \normo{\uk}_{\brk{\ep, \brk{1-4\ep}\half + 2\ep\brk{\frac{5}{14}-\ka}, 2\ep\brk{\frac{3}{7}-4\ka}+2\ep \brk{-\half}}} \\
		\lsm & \normo{\uk}_{\brk{0,\half,0}}^{1-4\ep} \normo{\uj}_{\brk{\half, \frac{5}{14}-\ka, \frac{3}{7}-4\ka}}^{2\ep}\normo{\uj}_{\brk{0,0,-\half}}^{2\ep}.
	\end{align*}
	Therefore, for $\brk{j,k}\in LL$ and $j\goe k$, we have
	\begin{align*}
		\normo{\Om(\uj,\uk)}_{\brk{\half, \frac{5}{14}-\ka, \frac{3}{7}-4\ka}} \lsm & \normo{\uj\uk}_{\brk{\half, \frac{5}{14}-\ka + \brk{1-4\ep}\half, \frac{3}{7}-4\ka+4[\brk{1-4\ep}\half]}} \\
		\sim & \normo{\uj\uk}_{\brk{\half, \brk{1-2\ep}\brk{\frac{5}{14}-\ka} + \brk{1-4\ep}\half + 2\ep\brk{\frac{5}{14}-\ka}, \frac{17}{7}-4\ka-8\ep}} \\
		\lsm & 2^{\half j+ \frac{1}{8}\ep k} 2^{-\half \be} \normo{\uj}_{\brk{\half-\ep, \brk{1-2\ep}\brk{\frac{5}{14}-\ka}, \frac{4}{7}}}\\ & \times\normo{\uk}_{\brk{\ep, \brk{1-4\ep}\half + 2\ep\brk{\frac{5}{14}-\ka}, -\frac{1}{8}\ep}}.
	\end{align*}
\end{proof}
\begin{lem}[Refined estimate for boundary term]
	Assume that $U$ and $U'$ are radial. For $0<\ka \ll \ep \ll 1$, 
	\begin{equation}
	\normo{\<D\>^{-1} \Om(U,U')}_{\WSI} \lsm 2^{-\be} \brk{\normo{U}_{\brk{0,\frac{1}{2},0|1}}\normo{U'}_{\WSI} + \normo{U}_{\WSI}\normo{U'}_{\brk{0,\frac{1}{2},0|1}}}.
	\end{equation}
\end{lem}
\begin{proof}
	This lemma is easy to obtain by H\"older inequality. In fact, for $j\loe -\be +10$, 
	\begin{align*}
		\normo{\uj}_{\brk{0,0,0}} \lsm  \normo{\uj}_{\brk{0,\half, 2}}
		\lsm 2^{j} 2^{-\be} \normo{\uj}_{\brk{0,\half, 0}},
	\end{align*}
	and for any $s>0$, $1\loe q,r\loe\I$,
	\EQ{
		\norm{\Om(U,U')}_{\brk{\rev q,\rev r, s}} \lsm \sum_{(j,k)\in LL} 2^{s\max\fbrk{j,k}} \norm{\uj\uk}_{L_t^q L_x^r}.
	}
\end{proof}
\begin{lem}[Trilinear term]
	Assume that $U$, $U''$ and $U'$ are radial. For $0<\ka \ll \ep \ll 1$, we have
	\begin{equation}
	\normo{\int_{0}^{t} K(t-s) \<D\>^{-1} \Om\brk{\<D\>^{-1}\brk{UU''},U'} ds}_{\SI} \lsm_{\be} \normo{U}_{\WSI}\normo{U''}_{\WSI}\normo{U'}_{\WSI}.
	\end{equation}
\end{lem}
\begin{proof}
	Since the output of $(UU'')_{HL+LH}$ has high frequency, we can divide the normal form into two parts
	\begin{equation*}
		\Om\brk{UU'',U'} = \Om\brk{\brk{UU''}_{LL},U'} + \Om\brk{\brk{UU''}_{HH},U'}.
	\end{equation*}
	By interpolation, for $j\loe -\be +10$, we have
	\begin{align*}
		\normo{\uj}_{\brk{\frac{1}{3},\frac{1}{6}, -\frac{13}{21}}} \lsm & \normo{\uj}_{\brk{\frac{1}{3},\frac{2}{3}\brk{\frac{5}{14}-\ka} + \brk{\frac{1}{3}-2\ep}\half, -\frac{13}{21} + 4[\frac{2}{3}\brk{\frac{5}{14}-\ka} + \brk{\frac{1}{3}-2\ep}\half-\frac{1}{6}]}} \\
		\lsm & \normo{\uj}_{\brk{\half, \frac{5}{14}-\ka, \frac{3}{7}-4\ka}}^{\frac{2}{3}} \normo{\uj}_{\brk{0,\half,0}}^{\frac{1}{3}-2\ep} \normo{\uj}_{\brk{0,0,-\half}}^{2\ep}.
	\end{align*}
	By boundeness of normal form, 
	\begin{align*}
		\normo{\Om\brk{\brk{UU''}_{LL},U'}}_{L^{1}_{t}L^{2}_{x}} \lsm & \normo{(UU'')_{LL}}_{L_t^{3/2}L_x^3}\normo{P_{\loe0}U'}_{L_t^3L_x^6} \\
		\lsm &\normo{P_{\loe0}U}_{L_t^3L_x^6} \normo{P_{\loe0}U''}_{L_t^3L_x^6} \normo{P_{\loe0}U'}_{L_t^3L_x^6}
	\end{align*}

	From the estimate of resonance term, we have
	\begin{equation*}
		\normo{P_{\loe -\be+10}\brk{UU''}_{HH}}_{\brk{1-2\ep,\half+2\ep,0}} \lsm \normo{U}_{\WSI}\normo{U''}_{\WSI}.
	\end{equation*}
	Therefore, 
	\begin{align*}
		\normo{\Om\brk{\brk{UU''}_{HH},\uk}}_{\brk{1-\ep,\half +3\ep, 10\ep}} \lsm & \normo{P_{\loe -\be+10}\brk{UU''}_{HH}}_{\brk{1-2\ep,\half+2\ep,0}} \normo{\uk}_{\brk{\ep,\ep,0}} \\
		\lsm & \normo{U}_{\WSI}\normo{U''}_{\WSI} 2^{-\half k} \normo{\uk}_{\brk{\ep,2\ep\brk{\frac{5}{14}-\ka} + \brk{1-4\ep}\half, 1}} \\
		\lsm & \normo{U}_{\WSI}\normo{U''}_{\WSI} 2^{-\half k} \normo{\uk}_{\brk{\half, \frac{5}{14}-\ka, \frac{3}{7}-4\ka}}^{2\ep} \\
		& \times \normo{\uk}_{\brk{0,\half,0}}^{1-4\ep} \normo{\uk}_{\brk{0,0,\half}}^{2\ep}.
	\end{align*}
\end{proof}
Combining all the above estimates, we obtain a perturbed Strichartz estimate in 4D case:
\begin{prop} \label{perturbed-4d}
	Let $d=4$, $0 < \ka \ll \ep \ll 1$ are small constants. Assume that $U$ is a solution of (\ref{mainequ}) with initial data $U_{0} \in \hra$ , then there exists $\th>0$ such that
	\begin{equation} \label{perturbed-4d-ineq}
	\normo{U}_{\SI} \lsm \normo{U_{0}}_{H^{1}} + 2^{-\th\be} \normo{U}_{\SI}^{2\brk{1-2\ep}} \normo{U}_{\ZI}^{4\ep} + \normo{U}_{\WSI}^{2} + \normo{U}_{\WSI}^{3},
	\end{equation}
	where the weak norm satisfies interpolation
	\EQ{
	\normo{U}_{\WSI} \lsm \norm{U}_{S(I)}^{1-2\ep} \norm{U}_{Z(I)}^{2\ep}.
	}
	Furthermore, we have small data scattering for (\ref{main}) in 4D case.
\end{prop}

\section{Variational analysis and Virial/Morawetz estimate}
We first review a classical result on the global well-posedness and blow-up dichotomy for Klein-Gordon equations with general nonlinearity $u^{p+1}$, which is due to Payne and Sattinger (see \cite{payne1975saddle}). Assume that $p\in\mathbb{N}_+$, and $u(t,x):\bbr\times\bbr^{d} \rightarrow \bbr$ is a solution of Klein-Gordon equation
\begin{equation} \label{general-equ}
\pdt^{2} u - \diff u + u = f(u),
\end{equation}
where $f(u) = u^{p+1}$. Define that $F(u) := \int f(u) \text{du}$ and $G(u) := uf(u) - 2F(u)$. The energy is
\begin{equation*}
	E(u(t),\ut(t)) = \int_{\bbr^{d}} \half |\pdt u(t,x)|^{2} + \half |\nabla u(t,x)|^{2} + \half |u(t,x)|^{2} -  \frac{1}{p+2} u(t)^{p+2} \dx.
\end{equation*}
Let $Q$ be the ground state, i.e. the unique radial solution to the elliptic equation
\EQ{
-\diff Q + Q = Q^{p+1},
}
and $Q$ is positive, smooth, and has exponential decay.
Define the stationary energy
\begin{equation*}
	J(\ph) = \half \normo{\nabla \ph}_2^2 + \frac{1}{2} \normo{\ph}_2^2 - \frac{1}{p+2} \int_{\bbr^d} \ph^{p+2} \dx,
\end{equation*}
where $\ph\in H^1$.
The potential well is
\begin{align*}
	j(\la) = \LL_{\al,\be} J(\ph) := & J\brk{e^{\al\la}\ph(e^{-\be\la}x)} \\
	= & \half e^{\brk{2\al+\brk{d-2}\be}\la} \normo{\nabla \ph}_2^2 + \frac{1}{2} e^{\brk{2\al+d\be}\la} \normo{\ph}_2^2 \\
	& - \frac{1}{p+2} e^{\brk{\brk{p+2}\al + d\be}\la} \int_{\bbr^d} \ph^{p+2} \dx.
\end{align*}
Define the sign functional $K_{\al,\be}(\ph) := \pd_{\la}|_{\la=0} \LL_{\al,\be} J(\ph)$, then
\begin{align*}
	K_{\al,\be}(\ph)  = &  \half \brk{2\al+\brk{d-2}\be} \normo{\nabla \ph}_2^2 + \frac{1}{2} \brk{2\al+d\be} \normo{\ph}_2^2 \\
	& - \frac{1}{p+2} \brk{\brk{p+2}\al + d\be} \int_{\bbr^d} \ph^{p+2} \dx.
\end{align*}
Minimal energy with respect to $K_{\al,\be}$ is
\begin{align*}
	m_{\al,\be} := \inf\{J(\ph):\ph\in H^1 \backslash \{0\}, K_{\al,\be}(\ph)=0\}.
\end{align*}
We take two subsets in energy space:
\EQ{
	\K_{\al,\be}^{+} := \{\brk{\uo,\ua} \in H^1\times L^2: E(\uo,\ua)<m_{\al,\be},\ K_{\al,\be}\brk{\uo} \goe0\}, \\
	\K_{\al,\be}^{-} := \{\brk{\uo,\ua} \in H^1\times L^2: E(\uo,\ua)<m_{\al,\be},\ K_{\al,\be}\brk{\uo} <0\}.
}
Now, with the above notations, we are prepared to state the dichotomy result:
\begin{thm}\label{PS-dichotomy}
	Let $f(u) = u^{2}$, and  $3\loe d\loe 5$. Assume further that $u(t)\in C\brk{I:H^1}$ is the solution to (\ref{main}) with initial data $u(0,x)=u_{0}$ and $u_{t}(0,x)=u_{1}$, where $I$ is the maximal lifespan interval. If $\brk{\uo,\ua}\in \K_{1,0}^+$ (or $\K_{1,0}^-$), then $u(t)$ is global (or blows up) in finite time, respectively. Furthermore, we have that $m_{1,0}=E(Q,0)$.
\end{thm}
\begin{proof}
	The proof is the same as that in \cite{payne1975saddle} and \cite{kg-focusing}, so we only sketch the proof. 
	
	The first step is to reduce to the positive case using the argument in \cite{payne1975saddle}. Let $\phi$ be the extremal of  $m_{1,0}$. We claim that $\phi$ cannot change sign. If $\phi$ changes sign, since
	\EQ{
	\int \abs{\nabla \abs{\phi}}^2 = \int_{\fbrk{\phi \ne 0}} \abs{\nabla \abs{\phi}}^2 = \int_{\fbrk{\phi \ne 0}} \aabs{\frac{\phi\nabla\phi}{|\phi|}}^2 =\int \abs{\nabla \phi}^2,
	}
	we have
	\EQ{
	0 = K_{1,0}(\phi) > K_{1,0}(|\phi|). 
	}
	Consider
	\EQ{
	\wt{j}(\la) = J(e^\la |\phi|) = \half e^{2\la} \int \abs{\nabla \phi}^2 + \half e^{2\la} \int \phi^2 -\rev 3 e^{3\la} \int |\phi|^3.
	}
	Then, there exists unique $\bar{\la}<0$ such that $\wt{j}'(\bar{\la}) = K_{1,0}(e^{\bar{\la}}|\phi|) = 0$. Thus, $K_{1,0}(e^{\bar{\la}}\phi) >0 $. Since $K_{1,0}(e^{\la}\phi) >0$ for $\la<0$, after integrating in $\la$, $J(e^{\bar{\la}}\phi) < J(\phi)$. We have that
	\EQ{
	J(e^{\bar{\la}}|\phi|) < J(e^{\bar{\la}}\phi) < J(\phi) =m_{1,0},
	}
	which contradicts to the definition of $m_{1,0}$.
	
	Furthermore, we can prove that
	\EQ{
	m_{1,0} = \inf\{J(\ph):\ph\in H^1 \backslash \{0\}, K_{1,0}(\ph)=0, \ph>0\}.
	}
	Let the right hand side of the above identity be $m'_{1,0}$. By definition, $m'_{1,0}\goe m_{1,0}$. Using similar argument of the preceding paragraph, $m'_{1,0}\loe m_{1,0}$. Since we consider the case $\ph>0$, the potential part can be viewed as $\int \ph^3 =\int |\ph|^3$. Thus, we can apply the method in gauge invariant case, see \cite{kg-focusing} for example. 
	
	Next, we can prove the Mountain Pass structure
	\EQ{
	m_{1,0} = \inf\{J(\ph) - \rev3 K_{1,0}(\ph):\ph\in H^1 \backslash \{0\}, K_{1,0}(\ph)\loe0, \phi>0\}.
	}
	Suppose that $K_{1,0}(\ph) < 0$ and $\ph>0$. Noting that $\int \ph^3>0$, consider the function
	\EQ{
	g(\la) = J(e^\la \ph) = \half e^{2\la} \int \abs{\nabla \ph}^2 + \half e^{2\la} \int \ph^2 - \rev 3 e^{3\la} \int \ph^3.
	}
	Since $g'(0) = K_{1,0}(\ph) < 0$, there exists $\la^*<0$ such that $g'(\la^*) = K_{1,0}(e^{\la^*} \ph) = 0$. Let 
	\EQ{
	G_0(\ph) = J(\ph) - \rev3 K_{1,0}(\ph) = \rev 6 \norm{\ph}_{H^1}^2.
	}
	We have $m_{1,0} \loe J(e^{\la^*}\ph) = G_0(e^{\la^*}\ph) < G_0(\ph)$. Therefore for $\ph>0$, when $K_{1,0}(\ph)<0$, $G_0(\ph)>m_{1,0}$, and when $K_{1,0}(\ph)=0$, $J(\ph) = G_0(\ph)$. Thus, the two infima coincide.
	
	Now, we can prove the existence of extremal of $m_{1,0}$. Let $\ph_n>0$ be some minimizing sequence, namely $K_{1,0}(\ph_n)\loe 0$ and $G_0(\ph_n)\ra m_{1,0}$. Using rearrangement function, we may assume that $\ph_n$ are radial. Using the compactness of embedding $H_{rad}^1 \subset L^3$, we can extract the limit $\ph_{\I}$ in $L^3$ up to subsequence. We also have $K_{1,0}(\ph_\I) \loe 0$, $J(\ph_\I) \loe m_{1,0}$ and $G_0(\ph_\I) \loe m_{1,0}$. 
	
	If $\ph_\I=0$, since $K_{1,0}(\ph_n)\loe 0$, $\int \ph_n^3 = \int |\ph_n|^3$ and strong convergence of $\ph_n$ in $L^3$, we have $\ph_n$ also converge strongly to $0$ in $H^1$. By Sobolev inequality
	\EQ{
		\norm{u}_{L^3}^3 \lsm \norm{u}_{H^1}^3,
	}
	$K_{1,0}(\ph_n) > 0$ for large $n$, which contradicts to the definition of $\ph_n$. Therefore $\ph_\I \ne 0$.
	
	If $K_{1,0}(\ph_\I) < 0 $, there exists $\la^*<0$ such that $K_{1,0}(e^{\la^*}\ph_\I) = 0$. By mountain pass, we have
	\EQ{
	m_{1,0} \loe G_0(e^{\la^*}\ph_\I) = e^{2\la^*} G_0(\ph_\I) < m_{1,0}.
	}
	Therefore, $K_{1,0}(\ph_\I) = 0$, $\ph_n$ converge strongly to $\ph_\I$ in $H^1$ and $J(\ph_\I)=m_{1,0}$.
	
	Let $\phi$ be the extremal, namely $\phi$ is radial, $J(\phi)=m_{1,0}$, $\phi>0$ and $K_{1,0}(\phi)=0$. By Lagrange multiplier method, there exists $\mu\in \bbr$ such that
	\EQ{
	- \De \phi +  \phi -  \phi^2 + \mu\brk{- 2\De \phi + 2\phi - 3\phi^2} = 0.
	}
	Then, multiply with $\phi$ and integrate in $x$. Noting that $\int \ph^3 = \int |\phi|^3$, we have
	\EQ{
		(1+2\mu)  \int \abs{\nabla \phi}^2 + (1+2\mu) \int \phi^2 - (1+3\mu) \int |\phi|^3 = 0.
	}
	Combining with $K_{1,0}(\phi)=0$, 
	\EQ{
	\mu \int |\phi|^3 =0.
	}
	We now have $\mu=0$, and that $\phi$ satisfies
	\EQ{
	-\De \phi + \phi = \phi^2.
	}
	By elliptic theory that $Q$ is the unique radial solution to the above equation in $H^1$, then we have $\phi = Q$ and $m_{1,0} = J(Q) = E(Q,0)$.

	We next prove that $\K_{1,0}^\pm$ is invariant under the flow \eqref{main}. Let $u(t,x)$ be the solution to \eqref{main} with maximal lifespan $I$. Assume that initial data $(u_0,u_1)\in\K_{1,0}^+$. If there exist $t*\in I$ such that $u(t^*)=0$, by Sobolev inequality
	\EQ{
	\norm{u}_{L^3}^3 \lsm \norm{u}_{H^1}^3.
	}
	we have $\norm{u(t)}_3^3 = o\brk{\norm{u}_{H^1}^2}$ near $t^*$, which implies $K_{1,0}(u(t)) \goe 0$. Noting that $E(u(t),u_t(t)) = E(u_0,u_1) < m_{1,0}$, the solution starting from $\K_{1,0}^+$ will remain in the same set, and so does the solution from $\K_{1,0}^-$.
	
	Now, we can prove the global-wellposedness of \eqref{main} in $\K_{1,0}^+$. Let $u(t,x)$ be the solution to \eqref{main} with maximal lifespan $I$. Assume that initial data $(u_0,u_1)\in\K_{1,0}^+$. Then, $K_{1,0}(u(t))\goe 0$ for all $t\in I$. We have
	\EQ{
	E(u,u_t) = & \half \int \abs{\nabla u}^2 + \half \int u^2 + \half \int u_t^2 - \rev 3\int u^3 \\
	\goe & \rev 6 \int \abs{\nabla u}^2 + \rev 6 \int u^2 + \half \int u_t^2.
	}
	Therefore, $E(u_0,u_1) \sim \norm{(u,u_t)}_{H^1 \times L^2}$. Combing with the local theory, we can extend the solution on $\bbr$.
	
	Next, we prove the variational estimate in $\K^-_{1,0}$. Let $J(\ph) < m_{1,0}$ and $K_{1,0}(\ph) < 0$. Observe that $\ph$ may change sign, but $\int \ph^3 \dx>0$. Consider
	\EQ{
	g(\la) = J(e^\la \ph) = \half e^{2\la} \int \abs{\nabla \ph}^2 + \half e^{2\la} \int \ph^2 -\rev 3 e^{3\la} \int \ph^3.
	}
	Since $g'(0) = K_{1,0}(\ph) < 0$, there exists $\la^*<0$ such that $g'(\la^*) = 0$. Noting that $\int \ph^3 \dx>0 $, we also have
	\EQ{
	g''(\la) = 2 e^{2\la} \int \abs{\nabla \ph}^2 + 2 e^{2\la} \int \ph^2 - 3 e^{3\la} \int \ph^3 \loe 2 g'(\la).
	}
	Integrating the above inequality on $[t^*,0]$, we have
	\EQ{
	g'(0) - g'(\la^*) \loe 2 (g(0) - g(\la^*)),
	}
	which gives the upper bound 
	\EQ{
	K_{1,0}(\ph) \loe - 2\brk{m_{1,0} - J(\ph)}.
	}

	Finally, we prove the blow-up result in $\K_{1,0}^-$ using the classical convex argument in \cite{payne1975saddle}. Let $u(t,x)$ be the solution to \eqref{main} with maximal lifespan $I$. Assume that initial data $(u_0,u_1)\in\K_{1,0}^-$. We consider $t>0$ case, and let $[0,+\I)\subset I$. Let $y(t)= \norm{u(t)}_{L^2}^2$. By equation \eqref{main},
	\EQ{
	\pdt^2 y(t) = & 2\brk{\norm{u_t}_2^2 - K_{1,0}(u(t))} \\
	= & 5\norm{u_t}_2^2 - 6E(u_0,u_t) + \norm{u}_{H^1}^2.
	}
	By variational estimate, we have $-K_{1,0}(u(t)) > C_1 $ for $t\goe 0$ and  some $C_1>0$. Then, $\pdt^2 y > 2C_1$. We have $\lim_{t\ra+\I} y(t) = +\I$. For large $t_0>0$ and $t>t_0$, we must have $\norm{u(t)}_{H^1}^2 > 6E(u_0,u_t)$. Thus, by Cauchy-Schwarz inequality,
	\EQ{
		\pdt^2 y(t) > 5\norm{u_t}_2^2 \goe \frac{5}{4} \frac{\brk{\pdt y(t)}^2}{y(t)},
	}
	or
	\EQ{
	\pdt^2(y^{-1/4}(t)) = - \frac{1}{4} y^{-9/4}\brk{yy''-\frac{5}{4}(y')^2} < 0.
	}
	Integrating on $[t_0,t]$, we have
	\EQ{
	\pdt(y^{-1/4})(t) \loe \pdt(y^{-1/4})(t_0),
	}
	and
	\EQ{
	y^{-1/4}(t) \loe \pdt(y^{-1/4})(t_0) (t-t_0) + y^{-1/4}(t_0).
	}
	Note that
	\EQ{
	\pdt(y^{-1/4})(t_0) = -\frac{1}{4} y(t_0)^{-5/4}y'(t_0)<0,
	}
	then for large $t>t_0$,
	\EQ{
	0\loe y^{-1/4}(t) \loe \pdt(y^{-1/4})(t_0) (t-t_0) + y^{-1/4}(t_0) < 0.
	}
	Therefore, the solution must blow up in finite time.
\end{proof}

\subsection{Variation in \texorpdfstring{$L^2$}{L squared}-critical case}
We consider another dichotomy below the ground state as follows:
\EQ{
	\K^{+} := \{\brk{\uo,\ua} \in H^1\times L^2: E(\uo,\ua)<E(Q,0),\ \normo{\uo}_2 < \normo{Q}_2\}, \\
	\K^{-} := \{\brk{\uo,\ua} \in H^1\times L^2: E(\uo,\ua)<E(Q,0),\ \normo{\uo}_2 > \normo{Q}_2\}.
}
In this subsection, we are going to review the result in \cite{killip2012scattering} that under the $L^2$-critical assumption, the solution starting from $\K^+$ exists globally, and the virial functional has a positive lower bound. Then, we will prove that the solution to 4D quadratic Klein-Gordon equation (\ref{main}) with initial data in $\K^-$ blows up in finite time. This blow-up result seems to be new, but the proof is essentially the same as that in \cite{ibrahim2014threshold}.

Now, suppose that $p=4/d$. Since $p$ is an integer, we have $d=4$ and $p=1$. We first recall the classical sharp Gagliardo-Nirenberg inequality:
\begin{prop}\cite{Nagy1941,Wei83CMP}
	For any $g\in H^1$, we have
	$$\int_{\R^d} |g(x)|^{\frac{2\brk{d+2}}{d}} \dx \loe \frac{d+2}{d} \brk{\frac{\normo{g}_{2}}{\normo{Q}_{2}}}^{\frac{4}{d}} \normo{\nabla g}_{2}^{2},$$
	where the equality holds if and only if $g(x) = \al Q\brk{\la\brk{x-x_{0}}}$ for some $\al\in \C$, $\la\in \brk{0,\I}$, and $x_0\in \bbr^d$. 
	Furthermore, suppose that a function $g(x)\in H^1$ satisfies $\normo{g}_{2} < \normo{Q}_{2}$, then we have
	\begin{equation}
	\normo{\nabla g}_{2}^{2} - \frac{d}{d+2} \int |g(x)|^{\frac{2\brk{d+2}}{d}}\dx \goe \brk{1 - \brk{\frac{\normo{g}_{2}}{\normo{Q}_{2}}}^{\frac{4}{d}}} \normo{\nabla g}_{2}^{2}.
	\end{equation}
\end{prop}
As a corollary, the inequality gives us an equivalent characterization for Virial functional 
\begin{equation*}
K(g) := \int |\nabla g|^2 - \frac{dp}{2\brk{p+2}} \int |u|^{p+2} = \int |\nabla g|^2 - \frac{d}{d+2} \int |g|^{\frac{2\brk{d+2}}{d}}.
\end{equation*}
In general, $\normo{g}_2 \loe \normo{Q}_2$ implies that $K(g)\goe 0$.
Recall the energy identity and the Pohozaev identity for the ground state $Q$, i.e.
\begin{align*}
	\normo{Q}_{H^{1}}^{2} =  \normo{Q}_{p+2}^{p+2},	
\end{align*}
and
\begin{equation*}
	\frac{d-2}{2} \normo{\nabla Q}_{2}^{2} + \frac{d}{2} \normo{Q}_{2}^{2} = \frac{d}{p+2}\normo{Q}_{p+2}^{p+2},
\end{equation*}
which imply
\begin{align*}
	E(Q,0) = & \frac{1}{2} \normo{Q}_{2}^2.
\end{align*}
Note that for 4D quadratic equation \eqref{main},
\EQ{
E(\uo,\ua) =& \half \normo{\nabla\uo}_2^2 - \rev 3 \int u_0^3 \dx + \half \normo{\uo}_2^2 + \half \normo{\ua}_2^2 \\
\goe & \half K\brk{\uo} + \half \normo{\uo}_2^2 + \half \normo{\ua}_2^2,
}
so we have
\begin{cor}
	Suppose that $E(u_{0},u_{1}) < E(Q,0)$, then $\normo{u_{0}}_{2} < \normo{Q}_{2}$ if and only if $K(\uo)>0$ or $\uo=0$.
\end{cor}
\begin{prop}
	Let $d=4$ and $p=1$. Assume that $\normo{u_{0}}_{2} < \normo{Q}_{2}$ and $E(u_{0},u_{1}) < E(Q,0)$. If $u(t,x)\in C\brk{I:H^{1}}$ is a solution of (\ref{main}) with initial data $u(0,x)=u_{0}$ and $u_{t}(0,x)=u_{1}$, for all $t\in I$, we have 
	\begin{align}
		 \normo{u(t)}_{2}  < A \normo{Q}_{2}, 
	\end{align}
	for some $A=A(E(\uo,\ut))= A(E(u(t),\ut(t)))<1$. Moreover,
	\begin{equation*}
		E(u(t),\ut(t)) \sim \normo{u}_{H^{1}}^{2} + \normo{\ut}_{2}^{2},
	\end{equation*}
	for all $t\in I$.
\end{prop}
\begin{proof}
	First, note that by Gagliardo-Nirenberg inequality, if for any $t\in I$, such that $\normo{u(t)}_{2}=\normo{Q}_{2}$, we must have $K(u(t))\goe 0$. Thus, \EQ{E(u(t),\ut(t)) \goe \half \normo{u(t)}_2^2 = \half \normo{Q}_2^2,} 
	which contradicts to our assumption. Therefore,
	\begin{align*}
	\normo{u\brk{t,\cdot}}_{2} < \normo{Q}_{2},
	\end{align*}
	or equivalently, $K(u(t)) \goe 0$
	for all $t\in I$. 
	
	Next, we are going to derive a gap between $\normo{u(t)}_{2}$ and $\normo{Q}_{2}$. 
	From the assumption, there exists a constant $A<1$ such that
	\begin{equation*}
	E(u(t),\ut(t)) < \frac{A^2}{2} \normo{Q}_{2}^2.
	\end{equation*}
	Therefore, $\normo{u(t)}_2 < A \normo{Q}_{2}$, for all $t\in I$. Using Gagliardo-Nirenberg inequality again, 
	\begin{equation*}
		 A \normo{\nabla u}_{2}^{2} \goe \frac{2}{3} \normo{u(t,x)}_{3}^{3}.
	\end{equation*}
	Therefore, we have
	\EQ{
	K(u(t)) \goe \brk{1-A} \normo{\nabla u(t)}_{2}^{2}
	}
	and
	\begin{align*}
		E(u(t),\ut(t)) \goe \half \brk{1-A} \normo{\nabla u(t)}_{2}^{2} +  \half \normo{\pdt u(t)}_{2}^{2}.
	\end{align*}
\end{proof}
From the local theory and the uniform bound of $\normo{u}_{H^{1}} + \normo{\pdt u}_{L^{2}}$, we obtain the global well-posedness for $L^2$-critical equation in $\K^+$. It follows from Theorem \ref{PS-dichotomy} that for 4D quadratic Klein-Gordon equation (\ref{main}), $\K^+ \subset \K_{1,0}^+$. 

Now, we can prove blow-up result in $\K^-$ for 4D quadratic Klein-Gordon equation (\ref{main}).
\begin{prop}
	Suppose that $d=4$. If $\normo{\uo}_{2} > \normo{Q}_2$ and $E(u_{0},u_{1}) < E(Q,0)$, the solution to (\ref{main}) blows up in finite time.
\end{prop}
\begin{proof}
The idea is to prove that $\K_{1,0}^\pm = \K^\pm$. It follows from $m_{1,0}=E(Q,0)$ that 
\EQ{
\K_{1,0}^+ \cup \K_{1,0}^- =\K^+ \cup \K^-.
}
Note that $\K_{1,0}^\pm$ are two disjoint sets, and so do $\K^\pm$. From the definition, we have that $\K_{1,0}^-$ and $\K^\pm$ are open sets. Since $\K^+ \subset \K_{1,0}^+$, it suffices to prove that $\K_{1,0}^+$ is connected.

%We can fix a $\ua$ such that $ \normo{\ua}_{2}^2/2 < J(Q)$ without losing of generality. Since the energy is below the ground state, $K_{1,0}(\uo) = 0$ implies $\uo=0$. Let $\ph$ be in some small neighbourhood of $0$ in $H^1$, i.e. $\normo{\ph}_{H^1} < \de$ for some small $\de>0$. By Gagliardo-Nirenberg inequality, we have
%\EQ{
%	\normo{\ph}_{p+2}^{p+2} \lsm \de^{\frac{d+4}{d}} \normo{\nabla \ph}_2^2,
%	}
%so $K_{1,0}(\ph) > 0$. This shows that $\K_{1,0}^+$ is open.

For any $\ph\ne 0$, such that $\brk{\ph,\ua}$ in $\K_{1,0}^+$, we define
\EQ{
	j_1(\la) := J(\la\ph) = \half \la^2 \brk{\normo{\nabla \ph}_2^2 + \normo{\ph}_2^2} - \frac{1}{3} \la^{3} \int \ph^3\dx,
	}
then $j'_1(\la) = \la \brk{\normo{\nabla \ph}_2^2 + \normo{\ph}_2^2} - \la^{2} \int \ph^3 \dx$. We have that $j'_1(1)\goe 0$.
If $0\loe \la \loe 1$, $j'_1(\la) \goe \la (1-\la) \brk{\normo{\nabla \ph}_2^2 + \normo{\ph}_2^2}\goe0$, so  $j_1(\la) \loe j_1(1) < J(Q) - \normo{\ua}_2^2/2$. Note that $K_{1,0}(\la\ph) = \la j'_1(\la) \goe 0$, so $\{\la\ph:0\loe\la\loe1\} \subset \K_{1,0}^+$. $\{\la \ph\}$ is a continuous orbit connecting $\ph$ and $0$, which implies that $\K_{1,0}^+$ is connected. Therefore the Proposition follows.
\end{proof}

\subsection{Virial/Morawetz estimate}
Take two functions $h(x):\bbr^d \ra \bbr^d$ and $q(x):\bbr^d\ra \bbr$. Let $h_j(x)$ be the $j$-th coordinate of the vector-valued function $h(x)$. After integrating by parts, we obtain the Morawetz identity for the general equation (\ref{general-equ}):
\EQ{
& -\pdt\brk{\int \ut\brk{h\cdot\nabla u + qu}\dx}\\
= & \sum_{j,k=1}^{d}\int \pdk u\pdk(h_j) \pdj u \dx + \half \int |u|^2 \brk{-\diff q(x)} \dx - \int q(x)G(u) \dx \\
& + \int \brk{q(x)-\half \dive h(x)} \brk{-\aabs{\ut}^2 + \aabs{\nabla u}^2 + \aabs{u}^2 - F(u)} \dx.
}
See \cite{nakanishi2001remarks} for a more general version of Morawetz identity for complex-valued solution.

Using suitable cut-off function, we are able to obtain a decay estimate for focusing equation in radial case, in the spirit of recent work \cite{dodson2017new-radial}:
\begin{prop}[Virial/Morawetz estimate] \label{vi-mo}
	Let $d=4$, $p = 1$, and $u(t,x)\in C\brk{\R:H^{1}}$ be a solution of (\ref{main}) with initial data $\brk{\uo,\ua}\in \hra \times \lra$. Suppose that $E:=E(\uo,\ua)>0$ and $\normo{u(t)}_{H^1} + \normo{\ut(t)}_2 \sim E$. If there exists $A<1$ such that $\normo{u(t)}_2 \loe A\normo{Q}_2$ for all $t\in \R$, we have
	\begin{equation*}
		\int_{T}^{2T} \int_{\aabs{x}\loe R} |u|^{3} \dx \dt \loe C(E,A) \brk{R + T R^{-3/2}},
	\end{equation*}
for any $T>0$, $R>0$.
\end{prop}
\begin{proof}
	Take a cut-off function $\chi(r)\in C_{0}^{\infty}\brk{[0,\infty)}$ such that $\chi(r)=1$ if $r\loe1$ and $\chi(r)=0$ if $r>2$.  $\chi(R^{-1}r)$ is denoted by $\chir(r)$. 
	Let $$\ph(r) = \int_{0}^{r} \chir^{2}(s) \ds$$ and $$\Psi(x) = \frac{x}{\aabs{x}} \ph(\aabs{x}).$$
	By simple computations, we have
	\begin{align*}
		\pdk \Psi_{j} = & \de_{jk} \frac{\ph(\aabs{x})}{\aabs{x}} + \frac{\xj\xk}{\aabs{x}^{2}} \brk{\ph'(\aabs{x})-\frac{\ph(\aabs{x})}{\aabs{x}}}, \\
		\dive\Psi = & \frac{d-1}{\aabs{x}} \ph(\aabs{x}) + \ph'(\aabs{x}), \\
		\diff \dive \Psi = & \ph'''(\aabs{x}) + \frac{2(d-1)}{|x|} \ph''(|x|) - \frac{(d-1)(d-3)}{|x|^2} \brk{\frac{\ph(|x|)}{|x|} - \ph'(|x|)}, \\
		\ph'(\aabs{x}) = & \chir^{2}\brk{\aabs{x}}.
	\end{align*}
	Let $h(x) = \Psi(x)$ and $q = \half \dive h(x)$ in the Morawetz identity, then
	\begin{align*}
		& -\pdt\brk{\int \ut \brk{\Psi\cdot \nabla u + \frac{1}{2}\dive(\Psi)u} } \dx \\
		= & \sum_{j,k=1}^{d}\int u_{k} \pdk \Psi_{j} u_{j} \dx - \frac{1}{4} \int \diff\dive\Psi |u|^{2} \dx + \int \frac{1}{2}\dive(\Psi) G(u)\\
		= & \int \brk{\frac{\ph}{\aabs{x}} \aabs{\nabla u}^2 + \brk{\ph' - \frac{\ph}{\aabs{x}}} \aabs{\frac{x}{|x|}\cdot \nabla u}^2} \dx + \int \brk{\frac{d-1}{2} \frac{\ph}{\aabs{x}} + \half\ph'} G(u) \dx\\
		& -\frac{1}{4} \int \brk{\ph'''(\aabs{x}) + \frac{2(d-1)}{|x|} \ph''(|x|) - \frac{(d-1)(d-3)}{|x|^2} \brk{\frac{\ph(|x|)}{|x|} - \ph'(|x|)}} |u|^2 \dx \\
		= & \int \ph' \aabs{\nabla u}^2 \dx + \frac{d}{2} \int \ph' G(u) \dx + \frac{d-1}{2} \int \brk{\frac{\ph}{\aabs{x}} - \ph'} G(u) \dx \\
		& -\frac{1}{4} \int \brk{\ph'''(\aabs{x}) + \frac{2(d-1)}{|x|} \ph''(|x|) - \frac{(d-1)(d-3)}{|x|^2} \brk{\frac{\ph(|x|)}{|x|} - \ph'(|x|)}} |u|^2 \dx.
	\end{align*}
	From the definition of $\ph$, we have that the cut-off function $\ph(r)/r-\ph'(r)=0$ if $r\loe R$, and
	$0\loe \ph/r-\ph' \loe R/r$
	if $r\goe R$. Therefore, we can estimate easily
	\EQ{
	& \aabs{\int \brk{\ph'''(\aabs{x}) + \frac{6}{|x|} \ph''(|x|) - \frac{3}{|x|^2} \brk{\frac{\ph(|x|)}{|x|} - \ph'(|x|)}} |u|^2 \dx} \loe \frac{1}{R^2} C(E).
	}
	The Virial/Morawetz quantity is denoted by \EQ{
		M(t):=-\brk{\int \ut \brk{\Psi\cdot \nabla u + \half \dive(\Psi)u} } \dx,
	}
	and it is easy to see that $|M(t)| \lsm R.$ Therefore, noting that $G(u) = p u^{p+2}/(p+2)$,
	\begin{align*}
		\pdt M(t) \goe & \int \chir^{2} \brk{|\nabla u|^{2} + \frac{d}{2} G(u)} \dx  - C \int \aabs{\frac{\ph}{\aabs{x}} - \ph'} |u|^{p+2} \dx - C(E)\frac{1}{R^{2}}.
	\end{align*}
	In order to deal with the main term, since $\normo{\chir u}_2 \loe \normo{u}_2 \loe A \normo{Q}_2$, by Gagliardo-Nirenberg inequality, 
	\EQ{
	\normo{\nabla \brk{\chir u}}_{2}^{2} - \frac{2}{3} \int \aabs{\chir u}^{3}\dx \goe \brk{1 - \frac{\normo{\chir u}_{2}}{\normo{Q}_{2}}} \normo{\nabla \brk{\chir u}}_{2}^{2}.
	}
	We also have radial Sobolev inequality
	$$ \normo{  \aabs{x}^{3/2}|u| }_{L_{x}^{\infty}(\R^4)}\loe C(E).$$
	Therefore,
	\begin{align*}
		\pdt M(t) \goe & \int \chir^{2} \brk{|\nabla u|^{2}- \frac{dp}{2\brk{p+2}}|u|^{p+2}} \dx \\
		& - C\int_{\aabs{x}\goe R} \frac{R}{r} |u|^{p+2} \dx - \frac{1}{R^{2}} C(E) \\
		\goe & \int  (|\nabla \brk{\chir u}|^{2}- \frac{2}{3}|\chir u|^{3}) \dx \\
		& - C\int_{\aabs{x}\goe R} |u|^{3} \dx - \frac{1}{R^{2}} C(E) \\
		\goe & C(A) \int  |\chir u|^{3}\dx  -  \frac{C(E)}{R^{3/2}}\int_{\aabs{x}\goe R} |u|^{2} \dx - \frac{1}{R^{2}} C(E), \\
	\end{align*}
	where we use the identity
	\EQ{
	\int \chir^{2} |\nabla u|^{2} \dx = \int |\nabla (\chir u)|^{2} \dx + \int \chir \diff(\chir) |u|^2 \dx
	}
	for the second inequality. 
	Integrate in $t$ on $[T,2T]$, then the Proposition follows.
\end{proof}
\begin{cor} \label{vi-mo cor}
	Let $d=4$. Suppose that $u$ is a radial solution of quadratic Klein-Gordon equation (\ref{main}), whose initial data $\brk{\uo,\ua}$ satisfies $\normo{u_{0}}_{2} < \normo{Q}_{2},E(u_{0},u_{1}) < E(Q,0).$
	Define that $E:=E(u_{0},u_{1})>0$, then for any $\epo>0$, $T>1$ and $\ta >0$, there exists $T_{0}=T_{0}(\epo,T,E)\goe T$, such that
	\begin{equation} \label{vi-mo esti}
		\int_{T_{0}}^{T_{0}+\tau} \int |u(t,x)|^{3}\dx\dt \loe \epo.
	\end{equation}
\end{cor}
\begin{proof}
	In this case, $G(u)=u^3/3$. First, by variation, it follows from $(\uo,\ua)\in\K^+$ that the assumptions in Proposition \ref{vi-mo} hold with $A=C(E)$. Taking $R= T^{2/5}$, we have
	\begin{equation*}
		\int_{T}^{2T} \int_{\aabs{x}\loe T^{\frac{2}{5}}} |u(t,x)|^{3} \dx \dt \loe C(E) T^{\frac{2}{5}},
	\end{equation*}
and then
	\begin{equation*}
	\int_{T}^{2T} t^{-1}\int_{\aabs{x}\loe t^{\frac{2}{5}}} |u(t,x)|^{3} \dx \dt \loe C(E) T^{-\frac{3}{5}}.
	\end{equation*}
	By radial Sobolev inequality, we have
	\EQn{
		\int_{\aabs{x} \goe |t|^{2/5}} |u(t,x)|^{3} \dx \loe & C |t|^{-\frac{3}{5}} \int_{\aabs{x} \goe |t|^{2/5}} |u(t,x)|^{2} |\aabs{x}^{\frac{3}{2}}u| \dx \\
		\loe & C(E) |t|^{-\frac{3}{5}}.
	}
	Therefore,
	\EQn{
	\int_{T}^{2T} t^{-1}\int |u(t,x)|^{3} \dx \dt \loe C(E) T^{-\frac{3}{5}}.
	}
	
	Summation the above integral over $[2^{k}T,2^{k+1}T]$ for $k=0,\ 1,\ 2,\ ...$  yeilds
	\begin{equation*}
	\int_{T}^{\infty} t^{-1}\int |u(t,x)|^{3} \dx \dt \loe C(E) T^{-\frac{3}{5}}.
	\end{equation*}
	For any fixed $\ta>0$, divide the above integral into $[T + k \ta, T + \brk{k+1} \ta]$ for non-negative integer $k$, i.e.
	\begin{equation*}
	\sum_{k\in \mathbb{N}} \frac{1}{T + \brk{k+1} \ta} \int_{T + k \ta}^{T + \brk{k+1} \ta} \int |u(t,x)|^{3} \dx \dt \loe C(E).
	\end{equation*}
	Since the series
	\begin{equation*}
	\sum_{k=0}^{+\I} \frac{1}{T + \brk{k+1} \ta}
	\end{equation*}
	diverges, there exists a $\Ta = T + k_0 \ta$ such that
	\begin{equation*}
	\int_{T_{0}}^{T_{0}+\tau} \int |u(t,x)|^{3}\dx\dt \loe \epo,
	\end{equation*}
	and the Corollary follows.
\end{proof}

\section{Large data scattering in 4D case}
\subsection{\texorpdfstring{$L^3$}{L cubic} decay after large time}
Corollary \ref{vi-mo cor} yields that localised $L_{x}^{3}$ norm of $u$ decays on arbitrarily large time interval, which is not sufficient for large data scattering. After normal form reduction, we need $L^{3}$ decay of $U= u-i\<D\>^{-1} \ut$ to establish the space-time bound. Now, we go back to the first order equation (\ref{1order}):
\begin{equation*}
	U(t,x) = K(t) U_{0}(x) - i \int_{0}^{t} K(t-s) \<D\>^{-1} \brk{u(s,x)^{2}} \ds.
\end{equation*}
\begin{prop} \label{l3-decay}
	Let $d=4$. Suppose that $u$ is a radial solution of quadratic Klein-Gordon equation (\ref{main}), whose initial data $\brk{\uo,\ua}$ satisfies $\normo{u_{0}}_{2} < \normo{Q}_{2},E(u_{0},u_{1}) < E(Q,0)$. For any $\epa>0$ and $T>0$, there exists $\taa=\taa(E,\epa)\goe C(E) \epa^{-5}$ and $T_{1} = T_{1}(E,\epa,T) $,  such that $T<T_{1}-\taa$, and
	\begin{equation}
		\sup_{t\in [T_{1}-\taa, T_{1}]} \normo{U(t,x)}_{L^{3}_{x}} \loe \epa.
	\end{equation}
\end{prop}
\begin{proof}
	Take a large constant $\taa>0$, and $R>0$ that will be defined later. We estimate $L_{x}^{3}$ norm of $U(t,x)$, and divide it into three parts
	\begin{align}
		\normo{U(t,x)}_{L_{x}^{3}} \loe & \normo{K(t) U_{0}(x)}_{3} \label{l3esti-1}\\
		& + \normo{\int_{0}^{t-\taa} K(t-s) \<D\>^{-1} \brk{u(s,x)^{2}} \ds}_{L_{x}^{3}} \label{l3esti-2}\\
		& + \int_{t-\taa}^{t} |t-s|^{-\frac{2}{3}} \normo{u(s,x) }_{L^{3}}^{2} \ds \label{l3esti-3}
	\end{align}
	
	First, we bound (\ref{l3esti-1}). Let $v(t):= K(t)U_{0}$. From radially improved Strichartz estimates, for any $2<q<3$, we have
	\begin{equation*}
		\normo{v(t)}_{L_{t}^{q}L_{x}^{3}} \loe C(E).
	\end{equation*}
	Note that $\normo{\pdt v(t,x)}_{L^2} \lsm \normo{v(t,x)}_{H^{1}_{x}}$, then $v(t)$ is Lipschitz continuous from $\bbr \rightarrow L^{3}_{x}$. Thus, we must have
	\begin{equation} \label{l3-1}
		\normo{K(t) U_{0}(x)}_{3} \rightarrow 0,
	\end{equation} 
	when $t \rightarrow \pm \infty$.
	
	Note that
	\begin{equation*}
		\int_{0}^{t-\taa} K(t-s) \<D\>^{-1} \brk{u(s,x)^{2}} \ds = K(t-t+\taa) U(t-\taa) - K(t) U_{0},
	\end{equation*}
	and for any small $0<\de\ll 1$,
	\EQn{
		& \normo{\int_{0}^{t-\taa} K(t-s) \<D\>^{-1} \brk{u(s,x)^{2}} \ds}_{B_{\I,2}^{-2 + \de}} \\ \loe & C \int_{0}^{t-\taa} |t-s|^{-2} \normo{ \<D\>^{-1} \brk{u(s,x)^{2}}}_{B_{1,2}^{1+\de}} \ds \\
		\loe & C \int_{0}^{t-\taa} |t-s|^{-2} \normo{u}_{H^1}^{2} \ds \loe C(E) \taa^{-1}.
	}
	We can also use the interpolation
	\EQn{
	\norm{\eqref{l3esti-2}}_{L_x^3} \lsm \norm{\eqref{l3esti-2}}_{B_{3,2}^{0}} \lsm \norm{\eqref{l3esti-2}}_{B_{\I,2}^{-2+\de}}^{1/3} \norm{\eqref{l3esti-2}}_{B_{2,2}^{1-\de/2}}^{2/3}.
	}
	Therefore, (\ref{l3esti-2}) can be bounded by
	\EQn{
		& \normo{\int_{0}^{t-\taa} K(t-s) \<D\>^{-1} \brk{u(s,x)^{2}} \ds}_{L_{x}^{3}} \\
		& \loe C \normo{\int_{0}^{t-\taa} K(t-s) \<D\>^{-1} \brk{u(s,x)^{2}} \ds}_{B_{\I,2}^{-2+\de}}^{\frac{1}{3}} \normo{K\brk{\taa} U\brk{t-\taa} - K(t) U_{0}}_{H^1}^{\frac{2}{3}} \\
		&\loe C(E) \taa^{-\frac{1}{3}}.
	}
	For any $\epa>0$ and $T>0$, there exists $\widetilde{T} = \widetilde{T}(\epa,T) > T$ and  $\taa=\taa(E,\epa)\goe C(E) \epa^{-5}$, such that for any $t>\widetilde{T}$, 
	\begin{equation*}
		\brk{\ref{l3esti-1}}+ \brk{\ref{l3esti-2}} \loe \frac{1}{2} \epa.
	\end{equation*}
	
	Finally, 
	\begin{align*}
		\int_{t-\taa}^{t} |t-s|^{-\frac{2}{3}} \normo{u(s,x)}_{L^{3}}^{2} \ds \loe & C(E) \int_{t-\taa}^{t} |t-s|^{-\frac{2}{3}} \normo{u(s,x)}_{L^{3}}^{\frac{1}{3}} \ds \\
		\loe & C(E) \taa^{\frac{2}{9}} \brk{\int_{t-\taa}^{t} \int |u(s,x)|^{3} \dx\ds}^{\frac{1}{9}}. 
	\end{align*}
	By Corollary \ref{vi-mo cor}, for the above $\widetilde{T}$, take $\ta = 2 \taa$ and $\epo\loe C(E) \taa^{-2} \epa^{9}$. Therefore, there exists $\widetilde{T}_{0} = \widetilde{T}_{0}(E,\epa,T)\goe\widetilde{T}$ such that for all $t \in [\widetilde{T}_{0} + \taa, \widetilde{T}_{0} + 2\taa]$, 
	\begin{equation*}
		\brk{\ref{l3esti-3}}\loe \frac{1}{2} \epa.
	\end{equation*}
	Now we take $T_{1} = \widetilde{T}_{0} + 2\taa$, and the Proposition follows.
\end{proof}

\subsection{Proof of Theorem \ref{thm-large}}
	Let $T_{2}>1$ and $\tab>0$ will be defined later. We consider the simplified equation with non-linear term $U^2$:
	\EQ{
	U(t,x) = K(t) U_{0}(x) - i \int_{0}^{t} K(t-s) \jb{D}^{-1} \brk{U(s,x)^{2}} \ds.
	}
	The equation after normal Form reduction can be rewritten as
	\begin{align}
	U(t,x) = & K(t)\brk{U_{0} + i\<D\>^{-1}\Om(U,U)\brk{0}} \label{int-equ 1}\\
	&-i\<D\>^{-1}\Om(U,U) \label{int-equ 2}\\
	& -i\int_{0}^{T_{2}-\tab} K(t-s) \<D\>^{-1}\brk{\brk{UU}_{LH+HL+HH} +2\Om\brk{-i\<D\>^{-1}U^{2},U}} \ds \label{int-equ 3}\\
	& -i\int_{T_{2}-\tab}^{T_{2}} K(t-s) \<D\>^{-1}\brk{\brk{UU}_{LH+HL+HH} +2\Om\brk{-i\<D\>^{-1}U^{2},U}} \ds \label{int-equ 4}\\
	& -i\int_{T_{2}}^{t} K(t-s) \<D\>^{-1}\brk{\brk{UU}_{LH+HL+HH} +2\Om\brk{-i\<D\>^{-1}U^{2},U}} \ds. \label{int-equ 5}
	\end{align}
	We are going to prove for any $\ep_1>0$, there exists $T_2$ such that
	\EQ{
	\norm{U}_{\WSN(T_2,+\I)} \lsm \ep_1^{3\ep^2/2}.
	}
	
	First, we have Strichartz bound
	\begin{equation*}
	\normo{K(t)\brk{U_{0} + i\<D\>^{-1}\Om(U,U)\brk{0}}}_{\WSN(\bbr)} \lsm \normo{U_{0}}_{H^{1}} + \normo{U_{0}}_{H^{1}}^{2},
	\end{equation*}
	then for any $\epa>0$, there exists $\widetilde{T}=\widetilde{T}(\epa)>0$, such that
	\begin{equation*}
	\normo{K(t)\brk{U_{0} + i\<D\>^{-1}\Om(U,U)\brk{0}}}_{\WSN\brk{T,+\infty}} \loe \epa,
	\end{equation*}
	for all $T >\widetilde{T}$. Thus, we take some $T_2>\widetilde{T}$.
	We also have refined bound
	\begin{equation*}
	\normo{\Om(U,U)}_{\WSN\brk{T_{2},\infty}} \loe 2^{-\beta} C(E) \normo{U}_{\WSN(T_2,+\I)}.
	\end{equation*}
	
	By interpolation, 
	\begin{equation*}
	\normo{f(x)}_{\ZI} \lsm \normo{f}_{L_{t}^{\infty}L_{x}^{3}}^{1-3\de} \normo{f}_{L_{t}^{\infty}H^{1}}^{3\de}.
	\end{equation*}
	Note we also have
	\begin{align*}
	\brk{\ref{int-equ 3}} = & -i[K(t-T_{2} +\tab)\brk{U(T_{2}-\tab) + i\<D\>^{-1}\Om(U,U)(T_{2}-\tab)} \\
	& - K(t)\brk{U_{0} + i\<D\>^{-1}\Om(U,U)(0)}].
	\end{align*}
	Therefore,
	\begin{equation*}
	\normo{\brk{\ref{int-equ 3}}}_{\WSN\brk{T_{2},+\infty}} \loe \normo{\brk{\ref{int-equ 3}}}_{S\brk{T_{2},+\infty}}^{1-2\ep} \normo{\brk{\ref{int-equ 3}}}_{\brk{0,0,-\half| -\frac{4}{3} +\de}}^{2\ep}.
	\end{equation*}
	Take $q$ such that $\frac{1}{q} = \frac{1}{8} + \frac{\ep}{4}$, then by Sobolev embedding, we have
	\EQn{
	& \normo{K(t-s) \<D\>^{-1}\brk{\brk{UU}_{LH+HL+HH} +2\Om\brk{-i\<D\>^{-1}U^{2},U}} }_{\brk{\dot{B}^{-\half}_{\infty}|\dot{B}^{-\frac{3}{4}+\de}_{\infty}}} \\
	\lsm &  \normo{ \normo{ K(t-s)\pk \brk{\brk{UU}_{LH+HL+HH} +2\Om\brk{-i\<D\>^{-1}U^{2},U}} }_{\dot{B}_{q}^{\ep}}}_{l^{2}_{k \loe 0}} \\
	& + \normo{ \normo{ K(t-s)\<D\>^{-1} \pk\brk{\brk{UU}_{LH+HL+HH} +2\Om\brk{-i\<D\>^{-1}U^{2},U}} }_{\dot{B}_{\infty}^{-\frac{4}{3} +\de}}}_{l^{2}_{k \goe 0}} \\
	\lsm & |t-s|^{-2\brk{1-\frac{2}{q}}} \normo{ 2^{\ep k} \brk{\normo{U}_{2q'}^{2}+ \normo{U}_{3q'}^{3}}}_{l^{2}_{k \loe 0}} \\
	& + |t-s|^{-2} \normo{2^{\brk{\frac{2}{3}+\de}k}\norm{\pk\brk{\brk{UU}_{LH+HL+HH}+2\Om\brk{-i\<D\>^{-1}U^{2},U}}}_{L_x^1}}_{l^{2}_{k \goe 0}} \\
	\lsm & |t-s|^{-2\brk{1-\frac{2}{q}}} + |t-s|^{-2},
	}
	where in the last inequality, we use $\pk\Om\brk{-i\<D\>^{-1}U^{2},U}=0$ for large $\be>0$, and bilinear estimate
	\EQn{
	\max_{k\goe 0}2^{\frac{5}{6}k} \norm{\pk\brk{(UU)_{HH+HL+LH}}}_{L^1} \lsm & \norm{P_{\goe 0}U}_{H^{11/12}}\norm{U}_{L^2} + \norm{P_{\goe0}U}_{H^{1/2}}^2 \loe C(E).
	}
	Thus, 
	\begin{align*}
	\normo{\brk{\ref{int-equ 3}}}_{\WSN\brk{T_{2},+\infty}} \loe & \normo{\brk{\ref{int-equ 3}}}_{S\brk{T_{2},+\infty}}^{1-2\ep} \normo{\brk{\ref{int-equ 3}}}_{\brk{0,0,-\half| -\frac{4}{3} +\de}}^{2\ep} \\
	\loe & C(E) \normo{\int_{0}^{T_{2}-\tab} \brk{|t-s|^{-2\brk{1-\frac{2}{q}}} + |t-s|^{-2}} \ds }_{L^{\infty}_{t\goe T_{2}}}^{2\ep} \\
	\loe & C(E) \tab^{-\ep+2\ep^2}.
	\end{align*}
	
	Next, we estimate (\ref{int-equ 4}). From the variation result, we have
	\EQ{
	\norm{U}_{L_t^\I H_x^1}^2 \sim E.
	}
	Therefore, using radial Strichartz estimate for \eqref{1order} and Sobolev embedding $H^1\subset L^4$, for any interval $I\subset \R$,
	\begin{equation}
	\normo{U}_{\SI} \loe  C(E) + C\norm{u^2}_{L_t^1 L_x^2(I\times \R^4)} \loe C(E) \jb{\aabs{I}}.
	\end{equation}
	We also have interpolation
	\begin{equation*}
	\normo{f}_{\brk{\dot{B}^{-\half}_{\infty}|\dot{B}^{-\frac{3}{4}+\de}_{\infty}}} \lsm \normo{f}_{L^{3}}^{1-3\de}  \normo{f}_{H^{1}}^{3\de}.
	\end{equation*}
	By Proposition \ref{l3-decay}, for $\epa>0$ and the above $\widetilde{T}$, there exists $\widetilde{\tau}_{1} = C(E) \epa^{-5}$ and $T_{2}$, such that
	\begin{equation*}
	\normo{U}_{L_{t}^{\infty}\brk{T_{2}-\widetilde{\ta}_{1},T_{2}:L_{x}^{3}}} \loe \epa.
	\end{equation*}
	Take $\tab = \epa^{-3\ep/2}$.
	Note that $[T_{2}-\tab,T_{2}] \subset [T_{2}-\widetilde{\ta}_{1},T_{2}]$ for sufficiently small $\epa$, then we have
	\begin{align*}
	\normo{\brk{\ref{int-equ 4}}}_{\WSN\brk{T_{2},+\infty}} \loe & C(E)\brk{ \normo{U}_{S\brk{T_{2}-\tab,T_{2}}}^{2-4\ep}\normo{U}_{Z\brk{T_{2}-\tab,T_{2}}}^{4\ep} +\normo{U}_{S\brk{T_{2}-\tab,T_{2}}}^{3-6\ep}\normo{U}_{Z\brk{T_{2}-\tab,T_{2}}}^{6\ep}} \\
	\loe & C(E)\brk{\<\tab\>^{2-4\ep} \normo{U}_{L_{t}^{\infty}\brk{T_{2}-\tab,T_{2}:L_{x}^{3}}}^{4\ep\brk{1-3\de}} + \<\tab\>^{3-6\ep} \normo{U}_{L_{t}^{\infty}\brk{T_{2}-\tab,T_{2}:L_{x}^{3}}}^{6\ep\brk{1-3\de}}} \\
	\loe & C(E)\brk{\<\tab\>^{2-4\ep} \epa^{4\ep\brk{1-3\de}} + \<\tab\>^{3-6\ep} \epa^{6\ep\brk{1-3\de}}} \\
	\loe & C(E) \epa^{\half \ep}.
	\end{align*}
	
	Above all, we have
	\begin{equation*}
	\normo{U}_{\WSN\brk{T_{2},+\infty}} \loe C(E)\brk{\epa + 2^{-\beta}\normo{U}_{\WSN\brk{T_{2},+\infty}} + \epa^{3\ep^2/2} + C(\beta)\normo{U}_{\WSN\brk{T_{2},+\infty}}^{2}+ \normo{U}_{\WSN\brk{T_{2},+\infty}}^{3}}.
	\end{equation*}
	Take a large $\beta = \be(E)>0$, such that
	$$C(E)2^{-\be} < \half.$$
	A standard bootstrap argument yields that for some $T_2 = T_2(\epa)$, $\normo{U}_{\WSN\brk{T_{2},+\infty}} \loe C(E) \epa^{3\ep^2/2}.$ Thus, we have  $\normo{U}_{\WSN\brk{0,+\infty}} \loe C(E)$.

	Finally, we can prove the large data scattering. We need to show that when $t \ra \pm\I$, $e^{-it\jb{D}}U(t)$ has limit in $H^1$. Note that 
	\EQ{
	K(-t)U(t) = &U_{0} + i\<D\>^{-1}\Om(U,U)\brk{0}-iK(-t)\<D\>^{-1}\Om(U,U) \\
	& -i\int_{0}^{t} K(-s) \<D\>^{-1}\Tres(U,U) \ds \\
	& -2i\int_{0}^{t} K(-s) \<D\>^{-1}\brk{\Om\brk{-i\<D\>^{-1}U^{2},U}} \ds.
	}
	By Strichartz estimates, we have
	\EQ{
	\normo{\int_{t_1}^{t_2} K(-s) \<D\>^{-1}\Tres(U,U) \ds}_{H^1} \lsm  \normo{U}_{\WSN\brk{t_1,t_2}}^2
	}
	and
	\EQ{
	\normo{\int_{t_1}^{t_2} K(-s) \<D\>^{-1}\brk{\Om\brk{-i\<D\>^{-1}U^{2},U}} \ds }_{H^1}
	\lsm  \normo{U}_{\WSN\brk{t_1,t_2}}^3.
	}
	By the boundedness of $\WSN$ norm, the integral term in $K(-t)U$ has limit in $H^1$. For the scattering, it suffices to prove that $\normo{K(-t)\<D\>^{-1}\Om(U,U)}_{H^1}$ tends to $0$, when $t\ra \pm\I$. We know that $P_{\loe 0}U \in \brk{1/3,1/6,-13/21}_{\bbr}$ from the boundedness of $\WSN$, then $\normo{P_{\loe 0}U}_{L_t^3\brk{\bbr:L_x^6}}\lsm 1$. Since $U(t)$ is bounded in $H^1$, we also have that
	\EQ{
	\normo{P_{\loe 0}U(t_1)-P_{\loe 0}U(t_2)}_{L_x^6} \lsm & \normo{P_{\loe 0}U(t_1)-P_{\loe 0}U(t_2)}_{L_x^2} \\
	\lsm & |t_1-t_2| \sup_{[t_1,t_2]} \normo{P_{\loe 0}\pdt U(t)}_{L_x^2} \\
	\lsm & |t_1-t_2| \sup_{[t_1,t_2]} \normo{P_{\loe 0}\brk{\jb{D}U + \jb{D}^{-1}\brk{U^2}}}_{L_x^2} \lsm |t_1-t_2|, \\
	}
	which implies that $\normo{P_{\loe 0}U(t)}_{L_x^6}$ is Lipschitz continuous in $t$. Thus, we obtain that 
	\EQ{
	\lim_{t\ra\pm\I} \normo{P_{\loe 0}U(t)}_{L_x^6} =0.
	}
	By the boundedness of Coifman-Meyer bilinear operator,
	\EQ{
	\normo{\Om(U,U)}_{H^1} \lsm \normo{P_{\loe 0} U}_{L^2} \normo{P_{\loe 0} U}_{L^6}.
	}
	Therefore, we have that $\lim_{t\ra\pm\I}\normo{\Om(U,U)}_{H^1} = 0$,  which completes the proof of Theorem \ref{thm-large}.

\section*{Acknowledgements}

The authors are grateful to Professor Baoxiang Wang and Professor Kenji Nakanishi for helpful discussions. Z. Guo is partially supported by ARC DP170101060. J. Shen was supported by China Scholarship Council.


\begin{thebibliography}{10}

\bibitem{brenner1984space}
P.~Brenner.
\newblock {On space-time means and everywhere defined scattering operators for
  nonlinear Klein-Gordon equations}.
\newblock {\em Mathematische Zeitschrift}, 186(3):383--391, 1984.

\bibitem{cazenave2003semilinear}
T.~Cazenave.
\newblock {\em {Semilinear schr{\"o}dinger equations}}, volume~10.
\newblock American Mathematical Soc., 2003.

\bibitem{delort2001existence}
J.-M. Delort.
\newblock {Existence globale et comportement asymptotique pour l'{\'e}quation
  de Klein--Gordon quasi lin{\'e}aire {\`a} donn{\'e}es petites en dimension
  1}.
\newblock {\em Annales scientifiques de l’Ecole normale sup{\'e}rieure},
  34(1):1--61, 2001.

\bibitem{delort2004global}
J.-M. Delort, D.~Fang, and R.~Xue.
\newblock {Global existence of small solutions for quadratic quasilinear
  Klein--Gordon systems in two space dimensions}.
\newblock {\em Journal of Functional Analysis}, 211(2):288--323, 2004.

\bibitem{dodson2015global}
B.~Dodson.
\newblock {Global well-posedness and scattering for the mass critical nonlinear
  Schr{\"o}dinger equation with mass below the mass of the ground state}.
\newblock {\em Advances in Mathematics}, 285:1589--1618, 2015.

\bibitem{dodson2016global}
B.~Dodson.
\newblock {Global well-posedness and scattering for the defocusing,
  $L^{2}$-critical, nonlinear Schr{\"o}dinger equation when $d=2$}.
\newblock {\em Duke Mathematical Journal}, 165(18):3435--3516, 2016.

\bibitem{dodson2017new-radial}
B.~Dodson and J.~Murphy.
\newblock {A new proof of scattering below the ground state for the 3d radial
  focusing cubic NLS}.
\newblock {\em Proceedings of the American Mathematical Society},
  145(11):4859--4867, 2017.

\bibitem{georgiev1996asymptotic}
V.~Georgiev and B.~Yordanov.
\newblock {Asymptotic behaviour of the one-dimensional Klein--Gordon equation
  with a cubic nonlinearity}.
\newblock {\em preprint}, 1996.

\bibitem{ginibre1985global}
J.~Ginibre and G.~Velo.
\newblock {The global Cauchy problem for the non linear Klein-Gordon equation}.
\newblock {\em Mathematische Zeitschrift}, 189(4):487--505, 1985.

\bibitem{glassey1973asymptotic}
R.~T. Glassey.
\newblock {On the asymptotic behavior of nonlinear wave equations}.
\newblock {\em Transactions of the American Mathematical Society},
  182:187--200, 1973.

\bibitem{guo2018scattering-GP}
Z.~Guo, Z.~Hani, and K.~Nakanishi.
\newblock {Scattering for the 3D Gross--Pitaevskii Equation}.
\newblock {\em Communications in Mathematical Physics}, 359(1):265--295, 2018.

\bibitem{guo2014generalized}
Z.~Guo, S.~Lee, K.~Nakanishi, and C.~Wang.
\newblock {Generalized Strichartz estimates and scattering for 3D Zakharov
  system}.
\newblock {\em Communications in Mathematical Physics}, 331(1):239--259, 2014.

\bibitem{guo2012small-za}
Z.~Guo and K.~Nakanishi.
\newblock {Small Energy Scattering for the Zakharov System with Radial
  Symmetry}.
\newblock {\em International Mathematics Research Notices}, 2014(9):2327--2342,
  2013.

\bibitem{guo2013global-za}
Z.~Guo, K.~Nakanishi, and S.~Wang.
\newblock {Global dynamics below the ground state energy for the Zakharov
  system in the 3D radial case}.
\newblock {\em Advances in Mathematics}, 238:412--441, 2013.

\bibitem{guo2014global-kgz}
Z.~Guo, K.~Nakanishi, and S.~Wang.
\newblock {Global dynamics below the ground state energy for the
  Klein-Gordon-Zakharov system in the 3D radial case}.
\newblock {\em Communications in Partial Differential Equations},
  39(6):1158--1184, 2014.

\bibitem{guo2012small-kgz}
Z.~Guo, K.~Nakanishi, and S.~Wang.
\newblock {Small energy scattering for the Klein-Gordon-Zakharov system with
  radial symmetry}.
\newblock {\em Mathematical Research Letters}, 21(4):733--755, 2014.

\bibitem{guo2014improved}
Z.~Guo and Y.~Wang.
\newblock {Improved Strichartz estimates for a class of dispersive equations in
  the radial case and their applications to nonlinear Schr{\"o}dinger and wave
  equations}.
\newblock {\em Journal d'Analyse Math{\'e}matique}, 124(1):1--38, 2014.

\bibitem{hayashi2009scattering}
N.~Hayashi and P.~I. Naumkin.
\newblock {Scattering operator for nonlinear Klein--Gordon equations}.
\newblock {\em Communications in Contemporary Mathematics}, 11(05):771--781,
  2009.

\bibitem{kg-focusing}
S.~Ibrahim, N.~Masmoudi, and K.~Nakanishi.
\newblock {Scattering threshold for the focusing nonlinear Klein-Gordon
  equation}.
\newblock {\em Anal. PDE}, 4(3):405--460, 2011.

\bibitem{ibrahim2014threshold}
S.~Ibrahim, N.~Masmoudi, and K.~Nakanishi.
\newblock {Threshold solutions in the case of mass-shift for the critical
  Klein-Gordon equation}.
\newblock {\em Transactions of the American Mathematical Society},
  366(11):5653--5669, 2014.

\bibitem{killip2012scattering}
R.~Killip, B.~Stovall, and M.~Visan.
\newblock {Scattering for the cubic Klein--Gordon equation in two space
  dimensions}.
\newblock {\em Transactions of the American Mathematical Society},
  364(3):1571--1631, 2012.

\bibitem{klainerman1985global}
S.~Klainerman.
\newblock {Global existence of small amplitude solutions to nonlinear
  klein-gordon equations in four space-time dimensions}.
\newblock {\em Communications on Pure and Applied Mathematics}, 38(5):631--641,
  1985.

\bibitem{kowalczyk2017kink}
M.~Kowalczyk, Y.~Martel, and C.~Mu{\~n}oz.
\newblock {Kink dynamics in the $\ph^4$ model: Asymptotic stability for odd
  perturbations in the energy space}.
\newblock {\em Journal of the American Mathematical Society}, 30(3):769--798,
  2017.

\bibitem{matsumura1976asymptotic}
A.~Matsumura.
\newblock {On the asymptotic behavior of solutions of semi-linear wave
  equations}.
\newblock {\em Publications of the Research Institute for Mathematical
  Sciences}, 12(1):169--189, 1976.

\bibitem{Nagy1941}
B.~V.~S. Nagy.
\newblock {{\"U}ber Integralungleichungen zwischen einer Funktion und ihrer
  Ableitung}.
\newblock {\em Acta Univ. Szeged. Sect. Sci. Math}, 10:64--74, 1941.

\bibitem{kg-subcritical-1-2}
K.~Nakanishi.
\newblock {Energy scattering for nonlinear Klein--Gordon and Schr{\"o}dinger
  equations in spatial dimensions 1 and 2}.
\newblock {\em Journal of Functional Analysis}, 169(1):201--225, 1999.

\bibitem{kg-nakanishi-energy-critical}
K.~Nakanishi.
\newblock {Scattering theory for the nonlinear Klein-Gordon equation with
  Sobolev critical power}.
\newblock {\em International Mathematics Research Notices}, 1999(1):31--60,
  1999.

\bibitem{nakanishi2001remarks}
K.~Nakanishi.
\newblock {Remarks on the energy scattering for nonlinear Klein-Gordon and
  Schr{\"o}dinger equations}.
\newblock {\em Tohoku Mathematical Journal, Second Series}, 53(2):285--303,
  2001.

\bibitem{nakanishi2008transfer}
K.~Nakanishi.
\newblock {Transfer of global wellposedness from nonlinear Klein-Gordon
  equation to nonlinear Schr{\"o}dinger equation}.
\newblock {\em Hokkaido Mathematical Journal}, 37(4):749--771, 2008.

\bibitem{ogawa1991blow}
T.~Ogawa and Y.~Tsutsumi.
\newblock {Blow-up of $H^1$ solution for the nonlinear Schr{\"o}dinger
  equation}.
\newblock {\em Journal of Differential Equations}, 92(2):317--330, 1991.

\bibitem{ozawa1996global}
T.~Ozawa, K.~Tsutaya, and Y.~Tsutsumi.
\newblock {Global existence and asymptotic behavior of solutions for the
  Klein-Gordon equations with quadratic nonlinearity in two space dimensions}.
\newblock {\em Mathematische Zeitschrift}, 222(3):341--362, 1996.

\bibitem{payne1975saddle}
L.~E. Payne and D.~H. Sattinger.
\newblock {Saddle points and instability of nonlinear hyperbolic equations}.
\newblock {\em Israel Journal of Mathematics}, 22(3-4):273--303, 1975.

\bibitem{schottdorf2012global}
T.~Schottdorf.
\newblock {Global existence without decay for quadratic Klein-Gordon
  equations}.
\newblock {\em arXiv preprint arXiv:1209.1518}, 2012.

\bibitem{shatah1985normal}
J.~Shatah.
\newblock {Normal forms and quadratic nonlinear Klein-Gordon equations}.
\newblock {\em Communications on Pure and Applied Mathematics}, 38(5):685--696,
  1985.

\bibitem{strauss1981nonlinear}
W.~A. Strauss.
\newblock {Nonlinear scattering theory at low energy}.
\newblock {\em Journal of functional analysis}, 41(1):110--133, 1981.

\bibitem{wang1998existence}
B.~Wang.
\newblock {On existence and scattering for critical and subcritical nonlinear
  Klein-Gordon equations in $H^s$}.
\newblock {\em Nonlinear Analysis: Theory, Methods \& Applications},
  31(5-6):573--587, 1998.

\bibitem{wang1999scattering}
B.~Wang.
\newblock {Scattering of solutions for critical and subcritical nonlinear
  Klein-Gordon equations in $ H^{s} $}.
\newblock {\em Discrete \& Continuous Dynamical Systems-A}, 5(4):753--763,
  1999.

\bibitem{Wei83CMP}
M.~I. Weinstein.
\newblock {Nonlinear Schr{\"o}dinger equations and sharp interpolation
  estimates}.
\newblock {\em Communications in Mathematical Physics}, 87(4):567--576, 1983.

\end{thebibliography}
\end{document}